\documentclass[11pt, letterpaper,reqno]{amsart}
\pdfoutput=1
\usepackage[utf8]{inputenc}
\usepackage{amsmath}
\usepackage{amssymb}
\usepackage{datetime}
\usepackage{fancyhdr}
\usepackage{graphicx}
\usepackage{latexsym}
\usepackage{pinlabel}
\usepackage{fullpage}
\usepackage{enumitem}
\usepackage{hyperref}
\usepackage{todonotes}

\hypersetup{
    unicode=true,          % non-Latin characters in Acrobat’s bookmarks
    pdftoolbar=true,        % show Acrobat’s toolbar?
    pdfmenubar=true,        % show Acrobat’s menu?
    pdffitwindow=false,     % window fit to page when opened
    pdfstartview={FitH},    % fits the width of the page to the window
    pdftitle={},    % title
    pdfauthor={Author},     % author
    pdfnewwindow=true,      % links in new PDF window
    colorlinks=true,       % false: boxed links; true: colored links
    linkcolor=red,          % color of internal links (change box color with linkbordercolor)
    citecolor=blue,        % color of links to bibliography
    filecolor=magenta,      % color of file links
    urlcolor=cyan,           % color of external links
}

\newtheorem{theorem}{Theorem}[section]
\newtheorem{lemma}[theorem]{Lemma}
\newtheorem{proposition}[theorem]{Proposition}
\newtheorem{corollary}[theorem]{Corollary}
\newtheorem{definition}[theorem]{Definition}

\newtheorem{remark}[theorem]{Remark}

\def \bdry {\partial}

\def \smskip{\vskip 1em}
\newcommand{\CC}{{\mathbb C}}

\newcommand{\p}{\partial}

\newcommand{\bit}{\begin{itemize}}
\newcommand{\eit}{\end{itemize}}

\newcommand{\ben}{\begin{enumerate}}
\DeclareRobustCommand{\een}{ \end{enumerate} }

\title{Singular fibers in algebraic fibrations of genus 2 and their monodromy factorizations}
\author[S.\ Sakall{\i}]{S\"umeyra Sakall{\i}}
\email{ssakalli@uark.edu}\urladdr{https://sites.google.com/umn.edu/ssakalli/home}
\address{Department of Mathematical Sciences, University of Arkansas, Fayetteville, AR 72701, USA}
\author[Jeremy Van Horn-Morris]{Jeremy Van Horn-Morris}\email{jv002@uark.edu}\urladdr{https://jv002.hosted.uark.edu}
\address{Department of Mathematical Sciences, University of Arkansas, Fayetteville, AR 72701, USA}
\begin{document}
\begin{abstract} 
Kodaira's classification of singular fibers in elliptic fibrations and its translation into the language of monodromies and Lefschetz fibrations has been a boon to the study of 4-manifolds. In this article, we begin the work of translating between singular fibers of genus 2 families of algebraic curves and the positive Dehn twist factorizations of Lefschetz fibrations for a certain subset of the singularities described by Namikawa and Ueno in the 70s. We look at four families of hypersurface singularities in $\mathbb{C}^3$. Each hypersurface comes equipped with a fibration by genus 2 algebraic curves which degenerate into a single singular fiber. We determine the resolution of each of the singularities in the family and find a flat deformation of the resolution into simpler pieces, resulting in a fibration of Lefschetz type. We then record the description of the Lefschetz as a positive factorization in Dehn twists. This gives us a dictionary between configurations of curves and monodromy factorizations for some singularities of genus 2 fibrations. 
\end{abstract}
\maketitle

\section{Introduction}

In \cite{Kodaira}, Kodaira classified all singular fibers in pencils of elliptic curves, and showed that in such a pencil, each fiber is either an elliptic curve, a rational curve with a node or a cusp, or a certain sum of rational curves of self-intersections $-2$. Such singular fibers come with a bevy of descriptions used to characterize them: numerical algebraic invariants, plumbing trees, and the monodromy around the link of the singular fiber, to name a few.

Kodaira's program was generalized by many mathematicians to tackle various aspects of this problem in higher genera. Iitaka and Ogg \cite{Ogg} made the first attempts to characterize the singular fibers in genus 2 fibrations. Then Namikawa and Ueno \cite{NamikawaUeno-list, NamikawaUeno-long} gave classification of the plumbing description of the central fiber and the complex representative of the central fiber in Siegel upper half space of singular fibers in pencils of genus two curves. They counted 120 distinct types of central fiber. Horikawa also attacked the problem \cite{Horikawa1, Horikawa} and constructed a numerical invariant of the singular fibers in the fibration.

To be precise, let $\pi: X \rightarrow \mathbb{D}$ be a complex algebraic family of complete curves of genus two over a disc $\mathbb{D}=\{t \in \mathbb C, |t|<\epsilon\}$, where $X$ is a minimal, non-singular, complex analytic surface, and $\pi$ is smooth over the punctured disc $D'=\mathbb{D}-\{0\}$ and has isolated, algebraic singularities. Thus, for every $t \in D'$ the fiber $\pi^{-1}(t)$ is a compact non-singular curve (Riemann surface) of genus two, the restriction of $\pi$ to $D'$ is a smooth fiber bundle, and the central fiber is a singular complex algebraic curve. Such families of curves are called degenerating families of algebraic curves. 

Atomic singular fibers are defined as singular fibers that cannot be split by any perturbation of the degenerating families into fibrations with multiple (simpler) singular fibers. Xiao and Reid proposed the problem of determining all the atomic singular fibers \cite{Xiao, Reid}, and in the genus two case it is studied in \cite{Horikawa, Xiao, Persson, ArakawaAshikaga}. In particular, in \cite{ArakawaAshikaga} and \cite{ArakawaAshikagaII} Arakawa and Ashikaga study splittings of hyperelliptic pencils. They show that any degeneration of hyperelliptic curves of genus two splits into three types namely; a genus one nodal curve (type $0_0$), two $-2$ curves of genera 0 and 1 transversally intersecting each other at two points (class I), and two $-1$ curves of genera both 1 transversally intersecting each other at one point (class II) (see Corollary 3.4 in \cite{ArakawaAshikaga}). Their splitting result comes from the method for a Morsification of singularities using A'Campo-Gusein Zade theory \cite{ACampo, Gusein}, in which a perturbation method is added along with their resolution process.

One invariant that can be extracted from the singularity is its \emph{monodromy}, that of the fibration around the central fiber thought of as an element of the mapping class group of a chosen nearby reference fiber. This is the starting point of the direction taken by Matsumoto and Montesinos who approached the problem from the topological perspective. A homeomorphism $f: \Sigma_g \rightarrow \Sigma_g$ is called pseudo-periodic if it takes a set $\mathcal{C}$ of disjoint union of simple closed curves (called cut curves) to itself, and on the complement it is isotopic to a homeomorphism of finite order, i.e., a periodic map. All monodromies of isolated singular fibers of algebraic families are pseudo-periodic. 
In \cite{MatsumotoMontesinos-paper, MatsumotoMontesinos-book}, Matsumoto and Montesinos show that in fact the monodromy is sufficient to recover Namikawa and Ueno's classification, at least up to homeomorphism. They study the Namikawa-Ueno type fibrations of genus $g\geq 2$, and they prove the following: the topological types of minimal degenerating families of Riemann surfaces of genus $g \geq 2$, over a disk, which are nonsingular outside the origin, are in a bijective correspondence with the conjugacy classes in the mapping class group $MCG(\Sigma_g)$ represented by pseudo-periodic maps of negative twist. The correspondence is given by the topological monodromy. (In the $g=1$ case, Kodaira showed that the analogous correspondence is surjective but not injective \cite{Kodaira}.) 

Throughout the work on this problem, mathematicians have worked to get better and stronger information about the singularity and the fibration. One of the more recent attempts at this comes from the work of Matsumoto building on his work with Montesinos and attempting to understand one of the singularities presented by Namikawa and Ueno. In \cite{Matsumoto-splitting} Matsumoto studies the singular fibration coming from a certain involution on $\Sigma_2$. The ultimate goal is to understand the deformation of this into a \emph{Lefschetz fibration}, the preferred description for smooth and symplectic topology \cite{Donaldson} and to present the corresponding positive factorization. He does this by an explicit deformation and uses computer calculations to compare the singular fibers, giving the positions of the corresponding vanishing cycles. In order to do this, rather than working with the initial holomorphic fibration he constructs a topologically equivalent one whose complex structures are not necessarily the same and works instead with that. This is particularly relevant to our work both in that we utilize related computer calculations but also do not address this particular example in this paper. 

In \cite{Ishizaka1}, Ishizaka begins with the ideas presented in \cite{MatsumotoMontesinos-paper}, ideas based in Nielsen and the topology of surfaces diffeomorphisms, to write down Dehn twist factorizations of all the finite order homeomorphisms of closed genus $g \geq 2$ surfaces. Indeed, these methods can be applied in all genera for hyperelliptic periodic maps which admit hyperelliptic deformations to Lefschetz fibrations (See also \cite{Ishizaka2}). Ishizaka in fact makes initial progress in determining the Lefschetz fibrations associated to some of the Namikawa-Ueno singularities, calculating the corresponding factorization corresponding to three of our examples (on the closed surface): $\phi_1, \phi_2$ and $\phi_3$, but addressing them in all genus. While not made explicit in the paper, the method presented there is essentially equivalent to the method we use in the paper here. By restricting to these three cases Ishizaka deals with only smooth surfaces or equisingular deformations and so Ishizaka's work avoids the need to invoke Laufer's work. 

In this paper, we follow a very similar outline of what Ishizaka does, considering the hyperelliptic quotient and the deformation of the branched curve inside $\CC^2$. Because we want  to deform the fibration itself, we have to take much more care in ensuring that we construct a flat deformation of the resolution of the explicit singularities in question. We restrict to the affine case and consider the polynomials used by Namikawa and Ueno in the subfamily of periodic homeomorphisms. First we construct explicit resolutions of all such singularities, then we construct deformations of the singular fibrations into Lefschetz fibrations, taking care to give explicit Dehn twists factorizations along with identifications of the reference fiber. We invoke Laufer to show that such deformations are flat. Our techniques are similar to Ishizaka in that we make extensive use of the hyperelliptic involution and pay considerable attention to the braid description of the branch locus. Indeed, there are branched coverings throughout the work. Finally, due to such care, these deformations are actually deformations of the underlying complex fibration on the resolution, and so one can read our theorem both as a characterization of the symplectic fibrations that underlie these algebraic families and the symplectic structures of the underlying algebraic surfaces which are more easily accessed via the plumbing diagrams that we construct in the resolutions of the singularities. 

These deformations give decompositions of the singularities into their atomic pieces. This process of deforming the singular fibration into a Lefschetz fibration is coined by Reid as Morsification \cite{Reid}, and so we give the explicit Morsifications of these singularities. This is natural in the symplectic setting as well, where Lefschetz fibrations are closely tied to both smooth Morse functions (through handle decompositions) and complex Morse functions (through the theory of Stein spaces).

In \cite{SS} and \cite{AS}, Akhmedov and the first named author worked with closed manifolds which are the total spaces of algebraic fibrations over $S^2$ with two singular fibers, say, $F_1, F_2$, where the singular fibers are in the list of \cite{NamikawaUeno-list}. For each pair of singular fibers in these algebraic fibrations, we reconstructed one fiber $F_1$ in a geometric way by finding the corresponding pencils of complex curves of genus two inside Hirzebruch surfaces. By blowing up these pencils we obtained specific types of Namikawa-Ueno's genus two singular fibers and sections precisely. In addition to constructing these singularities geometrically, we also introduced 2-nodal spherical deformations, by which we perturbed the dual fibers $F_2$ in the algebraic fibrations over $S^2$. Then by using them along with the symplectic surgeries (symplectic blow ups, symplectic resolutions, and generalized rational blowdowns) we constructed minimal symplectic 4-manifolds that are exotic copies of $\mathbb{CP}^2 \# 6\overline{\mathbb{CP}}^{2}$, $\mathbb{CP}^2 \# 7\overline{\mathbb{CP}}^{2}$, and $3\mathbb{CP}^2 \# k\overline{\mathbb{CP}}^{2}$ for $k=16, ..., 19$.

In this paper we consider algebraic families of genus 2 curves with one or two boundary components over $D^2$, and having one singular fiber. We construct the singular fibers and show that the algebraic fibrations we work with split into Lefschetz fibrations via a deformation of the complex structure. We find the monodromy factorizations of the total spaces that determine these Lefschetz fibrations, that is to say we determine the symplectic deformation types of the singular fibrations. We have also constructed 13 types of the Namikawa-Ueno fibers (Sections \ref{phi_1s} and \ref{phi_3s}), in a way that is motivated by, but different from, what was done in \cite{NamikawaUeno-long} and also different than our method in \cite{AS, SS}. % Let us give a brief outline of our work. 
To put in more detail, we work with polynomials $f(x,y,t)$ in $\mathbb{C}^3$ and we denote their zero sets by $V(f)$. Each of these algebraic varieties $V(f)$ has a fibration by algebraic curves with one singular fiber, more complicated than a Lefschetz singularity, and its generic fibers are smooth genus two curves with one or two boundary components. In general $V(f)$ is a singular variety which we resolve. We call the resolution space $X_f$. The fibration lifts to $X_f$ and the singular fiber lifts to its resolution graph. 

The purpose of the paper is to establish the correspondence between singular fibers of genus 2 algebraic fibrations, often thought of as configurations of complex curves via their resolution graphs, and their monodromy descriptions as Lefschetz fibrations.

\begin{theorem} \label{thm:main1}
The resolution $X_f$ of a singular algebraic variety $V(f)$ where $f$ is 
\begin{enumerate}
\item[i)]  $y^2 - x^5 - t^k$, $k = 1,\dots,10$,
\item[ii)]   $y^2 - x^6 - t^k$, $k = 1,\dots,6$,
\item[iii)]   $y^2 - x(x^4+t^k)$, $k = 1,\dots,8$,
\item[iv)]   $y^2 - x(x^5+t^k)$, $k = 1,\dots,10$,
\end{enumerate}
admits an (algebraic) fibration by hyperelliptic curves where all curves are smooth except when $t=0$. The central fiber is shown in Figures~\ref{phi1}, ~\ref{phi12} ~\ref{phi1ktable}, ~\ref{phi3}, ~\ref{phi3ktable}, \ref{phi2}, \ref{phi2ktable}, and \ref{phi4ktable}.
\end{theorem}

\begin{theorem} \label{thm:main2}
The genus 2 fibration on the resolution $X_f$ of a singular algebraic variety $V(f)$ where $f$ is 
\begin{enumerate}
\item[i)]  $y^2 - x^5 - t^k$, $k = 1,\dots, 6,8,9,10$,
\item[ii)]   $y^2 - x^6 - t^k$, $k = 1,2,4,5,6$,
\item[iii)]   $y^2 - x(x^4+t^k)$, $k = 1,\dots,8$,
\item[iv)]   $y^2 - x(x^5+t^k)$, $k = 1,\dots,10$,
\end{enumerate}
splits into a Lefschetz fibration described by one of the following positive words in the mapping class group of the genus 2 surface with either 1 or 2 boundary components:
\begin{center}
$\phi_1, \phi_2, \phi_3, \phi_4, \phi_{\tilde{2}}, \phi_{\tilde{4}}, I, \phi_A, \phi_B, \phi_1^2, \phi_1^3, \phi_2^2, \phi_3^2, \phi_4^2, \phi_{\tilde{2}}^2, \phi_{\tilde{4}}^2, \phi_1 I, \phi_2 I , \phi_{\tilde{2}} I , \phi_{\tilde{4}} I, \phi_3\phi_A, \phi_4\phi_B, \tau_{\bdry}$
\end{center}
where \begin{itemize}
\item $\phi_1 = \tau_1\tau_2\tau_3\tau_4$, 
\item $\phi_2 = \tau_1\tau_1\tau_2\tau_3\tau_4$, 
\item $\phi_{\tilde{2}} = \tau_5 \tau_{5'} \tau_4 \tau_3 \tau_2$,
\item $I$ is the hyperelliptic involution on the genus two surface with factorization $I = \tau_1\tau_2\tau_3\tau_4\tau_5\tau_{5'}\tau_4\tau_3\tau_2\tau_1$,
\end{itemize}
for the surface with one boundary component, and for the surface with two boundary components we have,
\begin{itemize}
\item $\phi_3 = \tau_1\tau_2\tau_3\tau_4\tau_5$, 
\item $\phi_4 = \tau_1\tau_1\tau_2\tau_3\tau_4\tau_5$, 
\item $\phi_{\tilde{4}} = \tau_{5'}\tau_5\tau_4\tau_3\tau_2\tau_1$
\item $\phi_A = \tau_1 \tau_4 \tau_3 \tau_{a_1} \tau_2 \tau_5 \tau_1 \tau_4 \tau_{b_1} \tau_{b_1'}$
\item $\phi_B = \tau_4 \tau_{a_2} \tau_3 \tau_5 \tau_2 \tau_4 \tau_{b_2} \tau_{b_2'}$, 
\end{itemize} 
and $\tau_\bdry$ stands for the boundary (multi-)twist on the surfaces $\Sigma_{2,1}$ and $\Sigma_{2,2}$ (whichever happens to be under consideration).
\end{theorem}

Here, we consider factorizations in the mapping class group $$\mathrm{MCG}(\Sigma_{2,\epsilon}) := \mathrm{Diff}^+(\Sigma_{2,\epsilon}, \bdry  \Sigma_{2,\epsilon})/\text{isotopy rel boundary},$$ where $\epsilon = \text{1 or 2}$ is the number of boundary components of the generic fiber. The Dehn twists $\tau_1, \cdots, \tau_5, \tau_{5'}$ are the standard generators of the hyperelliptic subgroup of the mapping class group of the genus two surface as shown in Figure \ref{fig:gens}. The curves $a_1$, $b_1$ and $b_1'$ from $\phi_A$ are shown in Figure~\ref{fig:phiAmonodromy} and the curves $a_2$, $b_2$ and $b_2'$ from $\phi_B$ are shown in Figure~\ref{fig:phiBmonodromy}. The monodromies $\phi_i$, $i=1,2,3,4$ are roots of the boundary multitwist of orders $10,8,6,5$ respectively.

\begin{remark}The apparent typos in the statement of Theorems~\ref{thm:main2} are deliberate omissions. There are two cases that we do not address in this theorem: $\phi_1^7, \phi_3^3$. The latter is related to the fibration constructed by Matsumoto \cite{Matsumoto-splitting} and the fibration on $y^2 = x^6 + t^3$ corresponds with some lift of Matsumoto's factorization to the genus 2 surface with 2 boundary components.\end{remark}

In genus 1, the dictionary between Kodaira's configurations of curves and monodromy factorizations was established by Harer-Kas-Kirby \cite{HarerKasKirby}. This has been hugely important to later efforts to produce exotic rational surfaces. Typical rational blowdown methods use configurations of curves and sections as constructed via factorizations of the monodromy of Lefschetz fibrations on starting manifolds like the elliptic surface $E(1)$, fibers like the $I_n$ and fishtail fibers in genus one fibrations (see e.g. \cite{PSS, StipSza, Ak}). 

Our theorem should be thought of as establishing this dictionary in genus 2. Given a Lefschetz fibration, finding a subword listed in Theorem~\ref{thm:main2} (or possibly more than one) shows us that we can find the corresponding configuration of curves in the manifold. 
For symplectic or smooth surgery constructions, one needs the knowledge of configurations of symplectic curves inside the starting 4-manifold. The starting point for many rational blowdown constructions of exotic rational surfaces is exactly the other language in the dictionary: monodromy factorizations. Begin with a known Lefschetz fibration on a useful manifold, use the factorizations of known singular fibers to produce specific configurations of curves, then use them in surgeries.

\begin{theorem} \label{thm:3} Given a Lefschetz fibration, a subfibration corresponding to a word in Theorem~\ref{thm:main2} contains the corresponding configuration of curves from Theorem~\ref{thm:main1}, and moreover we can deform the compatible symplectic structure to one in which the configuration of curves is symplectic.
%This would follow if the resolution were K\"ahler, as after a small enough deformation we'd have that the same symplectic form is compatible with the LF, and then we can use Gompf-Thurston.
\end{theorem}

\begin{corollary} \label{cor:1} Each of the monodromies in Theorem~\ref{thm:main2} is periodic (that is, some power the monodromy is some power of the boundary multitwist) and so the factorizations given are positive Dehn twist factorizations of roots of the boundary multitwist. Specifically we have 
\begin{itemize}
    \item $\phi_1^{10} = \tau_\bdry$,
    \item $\phi_2^8 = \tau_\bdry$.
    \item $\phi_3^6 = \tau_\bdry$, and
    \item $\phi_4^5 = \tau_\bdry$, 
\end{itemize}
where, as before, $\tau_\bdry$ is the boundary twist in the first two cases, and it is the boundary multitwist in the second two cases.
\end{corollary}

\noindent \textbf{Connections to page genus and the complexity of open books.}
Lastly, we point out that these examples are also interesting from the perspective of contact topology. Several of these resolutions are not plumbings along trees and so the only method for constructing a compatible open book is algebraic. For those resolutions that are plumbings of spheres along trees, these have at most one bad vertex. Trees with no bad vertices admit open book decompositions of genus 0 (\cite{Schonenberger}, \cite{EtguOz}) and trees with at most one bad vertex admit genus 1 open books (\cite{EtnyreOzbagci}). Interestingly, the Euler characteristic of the open books described by Theorem~\ref{thm:main2} are sometimes smaller than either the Euler characteristic of the open books described by \cite{Schonenberger}, \cite{EtguOz} or \cite{EtnyreOzbagci}. In that sense, some of the open books described by Theorem~\ref{thm:main2} have a smaller complexity than would be expected from the other constructions. As an example we look at $\phi_2^4$, and $\phi_2^6$, both of which are open books with a genus 2 page with one boundary component, so the Euler characteristic of the page is -3. Using the algorithm of \cite{EtnyreOzbagci}, in each of these cases the page of the open book associated to the plumbing tree for the resolution would have genus 1 and Euler characteristic -4.

\section{Acknowledgements}
We would like to thank Wenbo Niu, Lance Miller and Cagri Karakurt for useful discussions, and the Max Planck Institute in Bonn for supporting us during the start of our collaboration. We have used Mathematica for the braid movies. The second author was supported in part by Simons Foundation grant No.~639259. 

\section{Outline and proofs}
Since the proofs of Theorems~\ref{thm:main1} and \ref{thm:main2} are case by case analyses that make up most of the paper, we place the statements of their proofs here, along with a short outline of the paper.

First, in Section~\ref{sec:backg} we cover the necessary algebraic background along with the terminology for Lefschetz fibrations and a quick overview of their correspondence to branched covers of simply braided surfaces in $\CC^2$. We then calculate the resolution graphs of these singularities in Section \ref{Res}. We split the genus 2 fibrations on the resolutions into Lefschetz fibrations and record the corresponding Dehn twist factorizations in Sections \ref{sec:phi1and2}, \ref{sec:lower genus}, \ref{sec:split}.

\begin{proof}[Proof of Theorem~\ref{thm:main1}] The content of Section~\ref{Res} is the proof of Theorem~\ref{thm:main1}.\end{proof}

\begin{proof}[Proof of Theorem~\ref{thm:main2}] Using Theorem~\ref{thm:laufer}, the content of Sections~\ref{sec:phi1and2}, and \ref{sec:split} is the proof of Theorem~\ref{thm:main2}.\end{proof}

\begin{proof}[Proof of Theorem~\ref{thm:3}] Due to Gompf's work \cite{GompfLF} on Lefschetz fibrations generalizing Thurston's work \cite{ThurstonSF}, it's enough to show that the resolutions constructed in the proof of Theorem~\ref{thm:main1} and used in Theorem~\ref{thm:main2} admit a K\"ahler form (or really any symplectic form which is weakly compatible with the complex structure). To see that the resolution is K\"ahler, we embed the resolutions of the singular affine hypersurfaces of $\mathbf{C}^3$ into a compact algebraic surface. To do this, first complete the singular affine hypersurfaces described by one of the listed polynomials to a singular projective hypersurface in $\mathbb{CP}^3$. This will have an isolated singularity at the origin in one chart and potentially very complicated singularities at the $\mathbb{CP}^2$ at infinity. Then you resolve the singularities to get a smooth algebraic surface, which is necessarily K\"ahler. Within the affine chart, the only singularity is the isolated singularity at the origin, and we can choose the resolution above the affine chart to be the resolution constructed in Section~\ref{Res}. This gives an embedding of the resolution into a K\"ahler surface and hence the resolution admits a symplectic form which is compatible with the complex structure.

\end{proof}

\begin{proof}[Proof of Corollary~\ref{cor:1}] It suffices to prove each of the relations in the mapping class group of the relevant surface and since each of the monodromies is hyperelliptic, it's enough to look at the braids in the two-fold quotient. In \ref{sec:phi1}, we show that the monodromy $\phi_1$ around the link is the double branched cover of the braid corresponding to the (5,1) torus knot (so five strands and a 1/5 right-handed twist). The five-fold power of this is the full twist. Because there is an odd number of strands, the cover over the boundary is the nontrivial two fold cover and so the (5,5)-braid lifts to the hyperelliptic involution \cite{BirmanHilden}. Hence $\phi_1^{10} = \tau_\bdry$. For an analogous reason, $\phi_3$ is the double branched cover of the (6,1)-braid and so $\phi_3^6$ is the boundary multitwist. For the other two diffeomorphisms, it's easy to see that $\phi_2^8$ is Hurwitz equivalent to $\phi_1^{10}$ and similarly that $\phi_4^5$ is Hurwitz equivalent to $\phi_3^6$. \end{proof}

\section{Background}\label{sec:backg}
\subsection{Lefschetz fibrations}
  
Recall that a \emph{Lefschetz fibration} is a surjection $\pi : X \rightarrow B$, where $X$ has dimension 4 and $B$ has dimension 2, where the only allowed singularities have the local model of $\pi:(z_1, z_2) \rightarrow z_1^2 + z_2^2$ (in orientation-preserving complex coordinates). We often assume that all singularities lie in distinct fibers of $\pi$ (which can be realized by a small perturbation). When $X$ is either non-compact or with boundary we assume that the critical points of $\pi$ lie in the interior of $X$. The smooth fibers of the fibration are oriented surfaces and each singular fiber is a nodal singularity obtained by collapsing a simple closed curve (the \emph{vanishing cycle}) in a nearby smooth fiber. The \emph{monodromy} around a singular fiber is given by a positive (right-handed) Dehn twist along the vanishing cycle. We identify the vanishing cycles of different fibers by choosing a reference fiber (assumed to be smooth) $\pi^{-1}(x_0) = \Sigma_{g,k}$, a surface of genus $g$ and having $k$ boundary components (or ends if the surface is non compact), lying over the reference point $x_0 \in B$. Then for each singular fiber, we look at the fibration over a path from the reference fiber to the singular fiber to identify the vanishing cycle as a curve in $\pi^{-1}(x_0)$. The choice of paths gives a cyclic ordering of the vanishing cycles. For a genus $g$ Lefschetz fibration on a closed 4-manifold $X$ over $S^2$, the product of the positive Dehn twists for all vanishing cycles in the cyclic order induced by the choice of paths is equal to the identity element in the mapping class group MCG($\Sigma_g$) of the smooth fiber $\Sigma_g = \pi^{-1}(x_0)$. A Lefschetz fibration over $D^2$ is determined by a factorization of the monodromy of the fibration over $\bdry D^2$ as an element of MCG($\Sigma_{g,k}$) by positive Dehn twists. We call an ordered list of right handed Dehn twists in MCG($\Sigma_{g,k}$) a \emph{positive factorization}.\footnote{This is a factorization as we often think of the induced fibration over the boundary of the disk. This is a surface bundle over $S^1$ and so has a monodromy. The \emph{positive factorization} is a factorization of this monodromy.} Throughout this paper we will use braid notation for the factorizations, so that the cyclic ordering of the vanishing cycles yields a left-to-right ordering of the factorization.

Let $\tau_\gamma$ denote a Dehn twist around a loop $\gamma$ on the generic fiber $\Sigma_{g,k}$ of a Lefschetz fibration. Recall that \emph{Hurwitz moves} give a way of changing the local configuration of the arcs associated to two vanishing cycles which are adjacent in the cyclic ordering. For the corresponding positive factorization, we exchange the positions of two successive terms in either of the two ways that follows: $$\tau_\alpha \tau_{\beta} \rightarrow (\tau_\alpha \tau_{\beta}\tau_\alpha^{-1})\tau_\alpha = \tau_{\tau_\alpha^{-1} (\beta)} \tau_\alpha =: \tau_\alpha^{-1}(\tau_{\beta}) \tau_\alpha$$ or $$\tau_\alpha \tau_\beta \rightarrow \tau_{\beta}(\tau_{\beta}^{-1}\tau_{\alpha}\tau_{\beta}) = \tau_{\beta} \tau_{\tau_\beta(\alpha)} =: \tau_\beta \tau_\beta(\tau_\alpha)$$ in MCG($\Sigma_{g,k}$). 
We say that two positive factorizations are \emph{Hurwitz equivalent}, if they can be obtained from each other by a sequence of Hurwitz moves (\cite{Matsumoto96}, also \cite{Fuller1, Auroux, GS}).
    
There is an interplay between the mapping class groups of various surfaces and the braid groups with various numbers of marked points. Let $B_n$ denote the braid group on $n$ strands for $n \geq 2$. Recall that $B_n$ is generated by $n-1$ standard generators $a_1, \cdots, a_{n-1}$, where the braid relations $$a_i a_{i+1}a_i=a_{i+1}a_ia_{i+1}, \;i=1, \cdots, n-2, \;\;\; a_ia_j = a_ja_i, \;\text{if}\; |i-j|>1 $$ hold. A braid is called \emph{quasipositive} if it is the product of conjugates of the positive generators $a_i$ of the braid group $B_n$, i.e., it is represented by a quasipositive braid word $\prod_{k=1}^m w_k a_{i_k}w_k^{-1}$, where $w_1, \cdots, w_m$ are arbitrary words in $a_1, \cdots, a_{n-1}, a_1^{-1}, \cdots, a_{n-1}^{-1}$ \cite{Rudolph}. The subgroup of MCG($\Sigma_{g,\epsilon}$), where $\epsilon = 1, 2$, generated by the standard loops $\gamma_1, \cdots, \gamma_{2g+\epsilon-1}$ is the hyperelliptic subgroup of the mapping class group, and is related to the braid group $B_{2g+\epsilon}$ via the double branched cover, as follows. When we give the genus $g$ surface as a double branched cover of the disc branched in $2g + \epsilon$ points, the Dehn twists $\tau_1, \cdots, \tau_{2g+\epsilon-1}$ are the lifts of the standard generators of $B_{2g+\epsilon}$. Under this correspondence, a quasipositive braid factorization lifts to a positive Dehn twist factorization. An identification of the standard braid generators in the marked points then gives a set of Dehn twist generators of the hyperelliptic subgroup in the double branched cover. 

\subsection{Braided surfaces and branched covers}
We want to give a full symplectic description of a Lefschetz fibration associated to the resolution $X$ of some singular variety. For us, that means an ordered list of the \emph{vanishing cycles} of the Lefschetz fibration constructed via a flat deformation $\pi^s$ of some more complicated singular fibration $\pi$, along with an identification of the reference fiber $(\pi^s)^{-1}(t_0)$.

Our way of getting ahold of that information is to use a description of $X$ as the double branched cover over $C^2$ branched over some curve $C$ where the branch locus is a simply braided surface in $\CC \times \CC$.

A \emph{braided surface} in $\CC \times \CC$ is an embedded surface $\Delta$ which is transverse to the preimages of the projection on to the first factor $p_1: \CC \times \CC \rightarrow \CC$ except at finitely many points at which it is tangent. Throughout we will use coordinates $(t,x)$ on $\CC^2 = \CC \times \CC$. We say the braiding or branching at some point of tangency is $\emph{simple}$ if at the points of tangency 
$\Delta$ is locally parameterized as $(s^2,s)$ (in the $(t,x)$ coordinates). Often we perturb $\Delta$ so that each point of tangency occurs at different values of $t$. The \emph{index} of the braiding is the intersection number of the surface $\Delta$ with any of the fibers $\CC_t= p_1^{-1}(t)$ to which $\Delta$ is transverse. 

If $\Delta$ is a simply branched surface in $\CC^2$ and $D \subset \CC$ is a disk of large enough radius to enclose all the $t$-values of the points of tangency, then the double branched cover $Br: X=X_\Delta \to \CC^2$ is a smooth symplectic manifold and the composition $\pi := p_1 \circ Br: X \to D$ is a Lefschetz fibration. Above each value of $t$, the fiber $\Sigma_t$ is the double branched cover of the plane $\CC_t$ branched over the points $\Delta \cap \CC_t$. This is a surface with 1 or 2 boundary components (depending on whether the index of $\Delta$ is even or odd) and by Riemann-Hurwitz, the Euler characteristic of this double cover is $2 - \mathrm{index}(\Delta)$.

We can also determine the vanishing cycles of $\pi$. We start by choosing a reference fiber $t_0 \in \CC$ (often we will choose $t_0=1$) and paths $a_1, \dots, a_k$ from $t_0$ to the $t$-values of the points $t_1, \dots, t_k$ of tangency of $\Delta$ (respectively and missing all other tangency values). Above each of these paths, $\Delta$ intersects the fibers $\CC_t$ in a 1-parameter family of points which are disjoint except above the point of tangency $t_i$ where two points collide. Since all copies of $\CC_t$ are isomorphic via projection $p_2$ to the reference fiber $\CC_{t_0}$, over this path $\Delta$ is described by a movie of disjoint points in the $x$-plane ending with two of the points colliding. In the Lefschetz fibration $\pi$ on the double branched cover, the fiber above $t_i$ is a Lefschetz singularity and the vanishing cycle above $a_i$ can be recovered from this movie of $\Delta$. To find the vanishing cycle, rewind the movie slightly and indicate the path of the collision of the two points by an arc $q_i$. Generate the movie of $\Delta$ over $a_i$ by an isotopy of the $x$-plane and flow this arc back to the configuration above $t_0$ under the isotopy. Then the vanishing cycle of $\pi$ at $t_i$ over the arc $a_i$ is the preimage of $q_i$ under the double branched cover \cite{LoiP}. (See Figure~\ref{fig:braidqp}.)

\begin{figure}
    \centering
    \includegraphics[width = 6.7in]{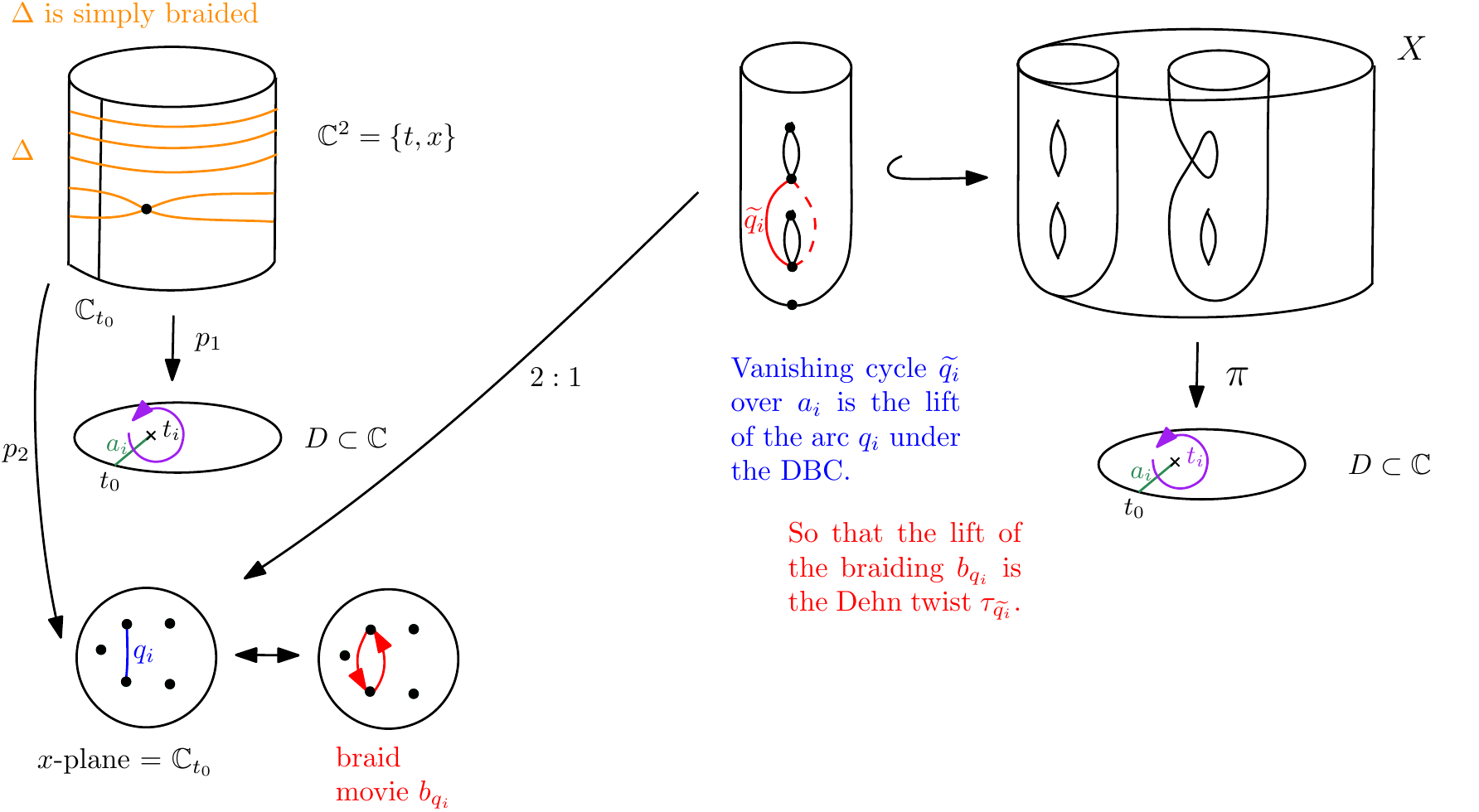}
    \caption{Braided surfaces and vanishing cycles in the Lefschetz fibrations}
    \label{fig:braidqp}
\end{figure}

The boundary of a braided surface is a braided link in $S^3$ and every simply braided surface with boundary $L$ corresponds to a quasipositive factorization of $L$ \cite{Rudolph} (up to band equivalence). The diagram of the quasipositive twists corresponds directly to the arcs described above. In the quasipositive factorization, each quasipositive twist is a braiding of two of the points around a choice of arcs. In the double branched cover, this braiding lifts to a right-handed Dehn twist. If the quasipositive twist comes from the movie of $\Delta$ over some path, then in the double branched cover the Lefschetz singularity has vanishing cycle the core of the corresponding Dehn twist.

To keep track of this picture, we often choose to characterize the branched cover $Br: \Sigma_{t_0} \mapsto \CC_{t_0}$ by identifying a chain of circles and arcs that form the \emph{skeleta} of $\Sigma_{t_0}$ and $\CC_{t_0}$ respectively and which are sent one to the other under the covering map. %(See Figure~\ref{fig:braid}.)

Throughout the next section, we give a description of a Lefschetz fibration in this language. We start with a smooth variety $X,$ which is the resolution of some hyperelliptic surface. That hyperelliptic surface is given as the branched cover of $\CC^2$ branched over some (not simply) braided surface. We deform this surface via some $\Delta_s$ which is simply braided for all $s\neq 0$ and verify that in the double branched cover $X_s$ this yields a flat deformation of $X$. We then recover the positive factorization that describes the Lefschetz fibration $\pi^s$ on $X_s$ using the descriptions of quasipositive braiding along certain arcs $q_i$, the corresponding skeleton in $\CC_{t_0}$, and their lifts to the reference fiber $\Sigma_{t_0}$.

\subsection{Algebraic preliminaries}

The focus of this article is the translation of understanding of a complex surface singularity between the description in terms of its resolution and a description using the language of Lefschetz fibrations. Throughout we will consider hypersurface singularities of the form $y^2 = x^5 + t^k$, $y^2 = x^6 + t^k$, $y^2 = x(x^4 + t^k)$ and $y^2 = x(x^5 + t^k)$. Notationally, we will use $f$ to refer to such a polynomial, $V(f)$ its vanishing locus in $\CC^3$, and $X_f$ the resolution of $V(f)$. Each of these singular varieties and their resolutions comes equipped with a fibration by projecting to the $t$-coordinate. The fibers are all hyperelliptic curves of bidegree $(2,5)$ or $(2,6)$. Topologically these are genus 2 surfaces with either 1 or 2 boundary components (respectively).

The choice of singularities is motivated by work of Namikawa and Ueno \cite{NamikawaUeno-list, NamikawaUeno-long} where they described the singularities that can occur in families of genus 2 curves. To avoid the kinds of algebraic issues that arose in the original article, we work in the affine setting. Following Namikawa and Ueno, we think of the variety $V(f)$ as a family of genus 2 curves parameterized by $t$ via the projection map onto the $t$ coordinate, $\pi_t: V(f) \to \CC$. This fibration persists to the resolution where the central fiber becomes the plumbing of curves associated to the resolution graph. First we will calculate the resolution of these singularities as described by a weighted plumbing tree and then we will construct a deformation of the genus 2 fibration on resolution into a Lefschetz fibration. 

We define a \emph{deformation} of a hypersurface singularity $V(f)$ to be a family of hypersurfaces $V(f_s)$ parameterized by $s\in \CC$, where $f_s$ is a polynomial in $(x,y,t)$ and $f_0 = f$ is the polynomial under consideration.

A \emph{splitting} of $f$ is a deformation $f_s$ so that the resolutions $X_f$ and $X_{f_s}$ lie in a flat family.

In general, given an arbitrary flat deformation $f_s$, the family of resolutions $X_{f_s}$ need not be flat and the different resolutions need not be even homeomorphic as manifolds. However, under certain constraints it is possible to guarantee flatness of the family $X_{f_s}$. 

Now consider $\pi:\mathcal{X}\mapsto B$ where $\mathcal{X}$ is a complex space (possibly singular), $B$ a complex manifold and for $t \in B$, $X_t := \pi^{-1}(t)$. 
\begin{theorem} (\cite{GK}, Satz 2.3) 
If $\pi: \mathcal{X} \rightarrow B$ is flat and $X_0$ is a complex manifold, then for every point $x_0 \in X_0$ there exist neighborhoods $U(x_0)$ in $\mathcal{X}$, $V(0)$ in $M$, $W(x_0)$ in $X_0$ and a biholomorphic map $\varphi: U \rightarrow V \times W$ with $pr_V \circ \varphi = \pi$, where $pr_V$ is the projection map onto $V$. 
\label{GK}
\end{theorem}

\begin{definition} A \emph{very weak simultaneous resolution} of a family $\lambda: \mathcal{V} \to T$ of singular hypersurfaces is a map $\Pi: \mathcal{M} \to \mathcal{V}$ such that 
    \begin{enumerate}
        \item $\Pi$ is proper.
        \item $\lambda \circ \Pi: \mathcal{M} \to T$ is flat. \label{ind1}
        \item $\Pi_s: M_s \to V_s$ is a resolution for all $s \in T$.
    \end{enumerate}
\end{definition}

Comparing with the notation used at the beginning of this subsection, $\mathcal{V} = \{V(f_s), s\in  T\}$, and the family $\mathcal{M}$ will come from the corresponding resolutions $M_s = X_{f_s}$. Since $M_0$ is smooth, flatness in \ref{ind1} means that $\lambda \circ \Pi$ is a (locally trivial) deformation of $M_0$ as a complex analytic and symplectic manifold (see Theorem \ref{GK} above).  
We will use such a resolution to \emph{split} the resolution of our singularities. The following theorem of Laufer gives the criterion we will use to show that the family of resolutions $\mathcal{X}$ above a given deformation $\mathcal{V}$ is indeed a splitting. 
   
    \begin{theorem}[Laufer \cite{Laufer-weak}, Theorem 5.7]  \label{thm:laufer} Let $\lambda: \mathcal{V} \to T$ be the germ of a flat deformation of the normal Gorenstein two-dimensional singularity $(V,p)$, with $T$ a reduced analytic space. Then $\lambda$ has a very weak simultaneous resolution, possibly after finite base change, if and only if $K_s \cdot K_s$ for $s \in T$ is constant. 
    \end{theorem}

We will use Laufer's theorem to construct a very weak simultaneous resolution. In each of our examples, we construct a flat deformation $V(f_s)$ of a two-dimensional singularity $V(f)$ and calculate $K_s \cdot K_s$ by its value on the corresponding resolution $X_{f_s}$. 

All of our singularities are hypersurface and hence also Gorenstein. Each $X_s$ comes with a fibration by algebraic curves induced by the projection onto the $t$ coordinate. Truncating to a disk of large radius gives the boundary the structure of an open book $B$ on $Y = \bdry X_s$, independent of $s$, and each $X_s$ is pseudo convex with boundary a plane field isotopic to the contact structure $\xi$ carried by $B$. For each $X_s$, then we have 
\[ K(X_s)^2=c_1^2(X_s) = 4 d_3(\xi)+2\chi(X_s) + 3 \sigma(X_s) \label{eqn:gompf}\]
where $\chi(X_s)$ is the Euler characteristic and $\sigma(X_s)$ is the signature, and $d_3$ is Gompf's 3-dimensional invariant of the plane field $\xi$ \cite{Gompf}. As $d_3$ is determined by $\xi$, it is independent of the choice of deformation. Thus a deformation determines a splitting precisely when the Euler characteristic and signature of $X_{f_s}$ is independent of $s$.

\begin{proposition} Let $(V,p)$ be an isolated two-dimensional singularity whose link is strictly pseudoconvex and with resolution $X$ which is symplectically deformation equivalent to a Stein domain. Let $V_s$ be a flat deformation of $(V,p)$ with isolated singularities and strictly pseudoconvex boundary and having minimal resolutions $X_s$, also symplectically deformation equivalent to a Stein domain. Then the family $X_s$ forms a flat deformation of $X$ (possibly after finite base change) if and only if the $X_s$ share the same $b_1$, $b_2^+$ and $b_2^-$ as $X$. 
\end{proposition}

\begin{proof}
If we assume the family $X_s$ is a flat deformation of $X$ (possibly after finite base change), then all fibers $X_s$ are diffeomorphic and hence all share the same values of $b_1$, $b_2^+$ and $b_2^-$. 

Conversely, assume that all fibers share the same values of $b_1$, $b_2^+$ and $b_2^-$. Since the deformation of $(V,p)$ is smooth away from the singularities, the contact structure on the pseudoconvex boundary remains constant and so all the terms in Gompf's $d_3$ formula, equation \ref{eqn:gompf} above, $d_3$, $b_1$, $b_2^+$, $b_2^-$, are constant and so $K_s \cdot K_s$ is constant as well. By Laufer, then, the family $X_s$ forms a very weak simulaneous resolution (possibly after finite base change) of $V_s$ and is, in particular, flat.\end{proof}

In addition, the fibration structure is induced by the common projection onto the $t$ coordinate, which persists after the resolution, and hence a simultaneous resolution additionally gives us a deformation of the corresponding fibration. In each case, we will (eventually) deform entirely into a Lefschetz fibration.

\section{Resolutions}
\label{Res}
In this section we resolve the singularities, which we will split below, and give their resolution graphs. 

\noindent {\bf Reconstructing the Namikawa-Ueno Fibers}

In Subsections~\ref{phi_1s} and \ref{phi_3s} we construct 13 types of the Namikawa-Ueno fibers, in a different way than in \cite{NamikawaUeno-long}. Namely, we construct them via Nemethi's algorithm in \cite{Nemethi-algorithm}, Section 3.6 and Theorem 3.7 a) (see also \cite{Nemethi-lectures}, III. Appendix 1).
\subsection{\texorpdfstring{{\boldmath $\phi_1^k$: $y^2 = x^5+t^k$}}{Phi1K}}
\label{phi_1s}
We start with resolving the plane curve singularity $x^5-y^2 =0$ by successive blow-ups. This gives us a tree $\mathcal{T}$ with four vertices $w_i$'s, denoting the 2-spheres coming from the exceptional divisors, whose multiplicities are 5, 10, 4, 2 and self intersections are $-2,-1,-3,-2$, respectively. (One can verify this graph via MAGMA). Each sphere $w_i$ intersects $w_{i-1}$ and $w_{i+1}$, and at the vertex of multiplicity 10 we have an arrow which denotes the proper transform of the starting singular curve. See the first step of Figure~\ref{phi1}, where we denote the multiplicities in parentheses.

\noindent {\boldmath  $k=1$}: 
To find the resolution of $y^2 = x^5+t$, to the tree $\mathcal{T}$ above we apply the Nemethi's algorithm in \cite{Nemethi-algorithm}, Section 3.6 with $N=1$. Since degree of the cover of $\mathcal{T}$ (which is the power of $t$) is $N=1$, we simply start with $\mathcal{T}$ and successively blow down the $-1$ curves. This gives us the $(2,5)$ cusp, hence we obtain the fiber of type VIII-1 in the Namikawa-Ueno's list in \cite{NamikawaUeno-list}. See Figure~\ref{phi1}, where after the first step we use the dual graph notation and denote the spheres by line segments/curves.
\begin{figure}[htb]
%\centering
{\includegraphics[width=15cm]{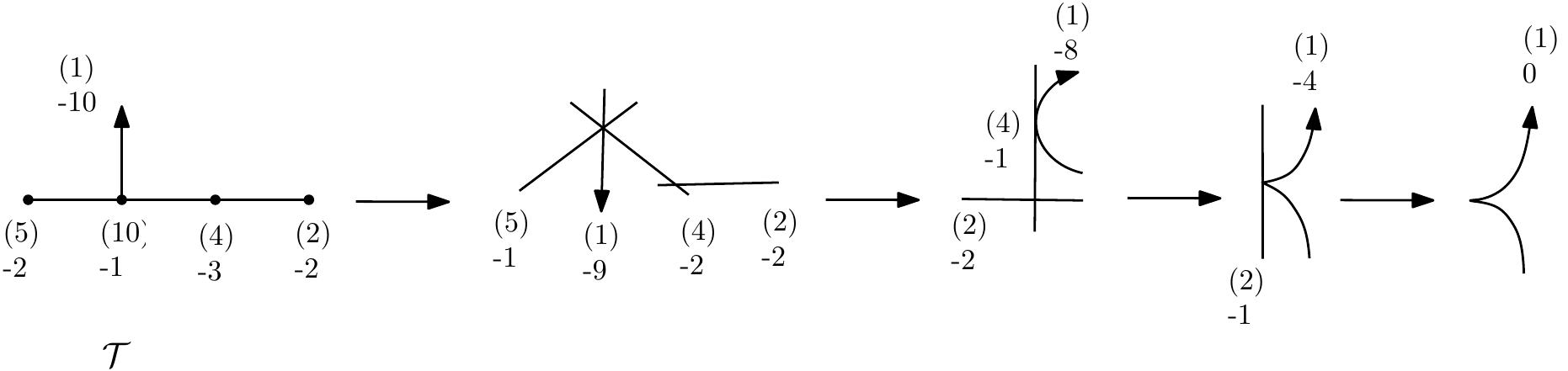}}
\caption{$\phi_1$ - Fiber VIII-1}
\label{phi1}
\end{figure}

\noindent {\boldmath  $k=2$}: 
To find the resolution of $y^2 = x^5+t^2$, to the tree $\mathcal{T}$ above, we apply the Nemethi's algorithm with $N=2$. Namely, we first find the preimages of each vertex of $\mathcal{T}$ under the degree two cover. Let us call the vertices of multiplicities $5,10,4,2$ in the tree $\mathcal{T}$, $w_1, \cdots, w_4$, respectively. In the algorithm $s_{w_i}$ denotes the number of all vertices and arrows adjacent to ${w_i}$. If $\{u_1, \cdots, u_s\}$ is the set of the neighbors of $w_i$, then $d_{w_i}$ denotes the $\text{gcd}(m_{w_i}, m_{u_1}, \cdots, m_{u_s})$, where $m$ of a vertex is its multiplicity. If the degree of the cover is $N$, above a vertex $w_i$, there are $\text{gcd}(d_{w_i}, N)$ vertices, each with multiplicity $\displaystyle{\frac{m_{w_i}}{(m_{w_i}, N)}}$. In our case, from $\mathcal{T}$, we have $s_{w_1}=1$, $s_{w_2}=3$, $s_{w_3}=2$ and $s_{w_4}=1$; and $d_{w_1}=5$, $d_{w_2}=1$, $d_{w_3}=2$ and $d_{w_4}=2$. For each of the preimages of $w_1$ and $w_2$, we find one vertex of multiplicity 5. In the preimage of $w_3$ there are two vertices of multiplicities 2, and in the preimage of $w_4$ there are two vertices of multiplicities 1. We also compute that all these vertices have genus 0 (see the genus formula in Step 2 in \cite{Nemethi-algorithm}, p.114 or Step 1 in \cite{Nemethi-lectures}, III. Appendix 1). Next, we find the preimages of all the edges and the arrow in $\mathcal{T}$ (where the arrow is the proper transform of the starting singular curve in the plane) as follows. From the algorithm we find that above the edge $(5,10)$ whose end points have weights $5, 10$ we have $\text{gcd}(5,10,2)=1$ string of type $G(5,10,2)$  (see \cite{Nemethi-lectures}, III. Appendix 1 for the string notation). Thus we solve $10+x \cdot 5 \equiv 0 \; (\text{mod} \; 2)$. Since $0 \leq x_1=0 <2$ is the solution, this gives us that in the string there are no vertices between the end points. Hence, above the edge $(5,10)$, there is one edge whose end points have weights $5,5$. In the same way, above the edge $(10,4)$ there are $\text{gcd}(10,4,2) = 2$ copies of strings of type $G(5,2,1)$. Then we solve $2+x \cdot 5 \equiv 0 \; (\text{mod} \; 1)$. Since the solution is $x_1=0$, in each of the two strings there are no vertices between their end points. Hence we get 2 edges whose initial vertices is common and has multiplicity 5, and their end points both have multiplicities 2. Above the edge $(4,2)$ there are $\text{gcd}(4,2,2)=2$ copies of strings of type $G(2,1,1)$. We solve $1+x \cdot 2 \equiv 0 \; (\text{mod}\; 1)$, and find that there are no additional vertices in the strings. Hence above the edge $(4,2)$ we find 2 edges each of whose end points have multiplicities $2$ and $1$. Finally, above the arrow in the tree $\mathcal{T}$, there is one string of type $G(10,1,2)$ whose one end point is the vertex of multiplicity 5, and the other end is the arrowhead. We solve $1+x \cdot 5 \equiv 0 \; (\text{mod}\; 1)$, since $x_1=0$ we have no additional vertices between the vertex of multiplicity 5 and the arrowhead. Then, from the algorithm we compute the self intersections of all the vertices (see formula $(*)$ on p.13, \cite{Nemethi-lectures}). This gives us the graph as shown in part a) of Figure~\ref{phi12}. After blowing down the $-1$ spheres we obtain the configuration as in part b), Figure~\ref{phi12}. We note that this is exactly the singular fiber of type IX-1 in the Namikawa-Ueno's notation in \cite{NamikawaUeno-list}. We also remark that Namikawa-Ueno obtain these genus two singular fibers in \cite{NamikawaUeno-long} by orbifold quotients, whereas here we construct them in a different way.
\begin{figure}[htb]
%\centering
{\includegraphics[width=9cm]{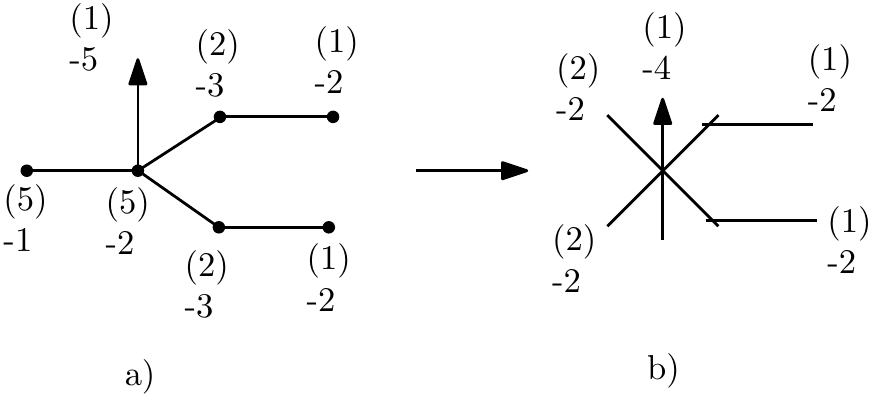}}
\caption{$\phi_1^2$ - Fiber IX-1}
\label{phi12}
\end{figure}

\noindent {\boldmath $k=3, \cdots, 10$}: 
To find the resolutions of $y^2 = x^5+t^k$, $k=3, \cdots, 10$, we follow the same steps as in the $k=2$ case above. To the tree $\mathcal{T}$, we apply the Nemethi's algorithm with $N=k$. We summarize them in Figure~\ref{phi1ktable} where the second column consists of the graphs after applying Nemethi's algorithm with all the self-intersections and multiplicities of the components. Let us note that in the last row we have that the central vertex has genus two. The third column of the table shows the resulting graphs after consecutive blow-downs. In the third column we also note the types of the fibers as in the Namikawa-Ueno's notation.
\begin{figure}[htb]
\centering
{\includegraphics[width=15cm]{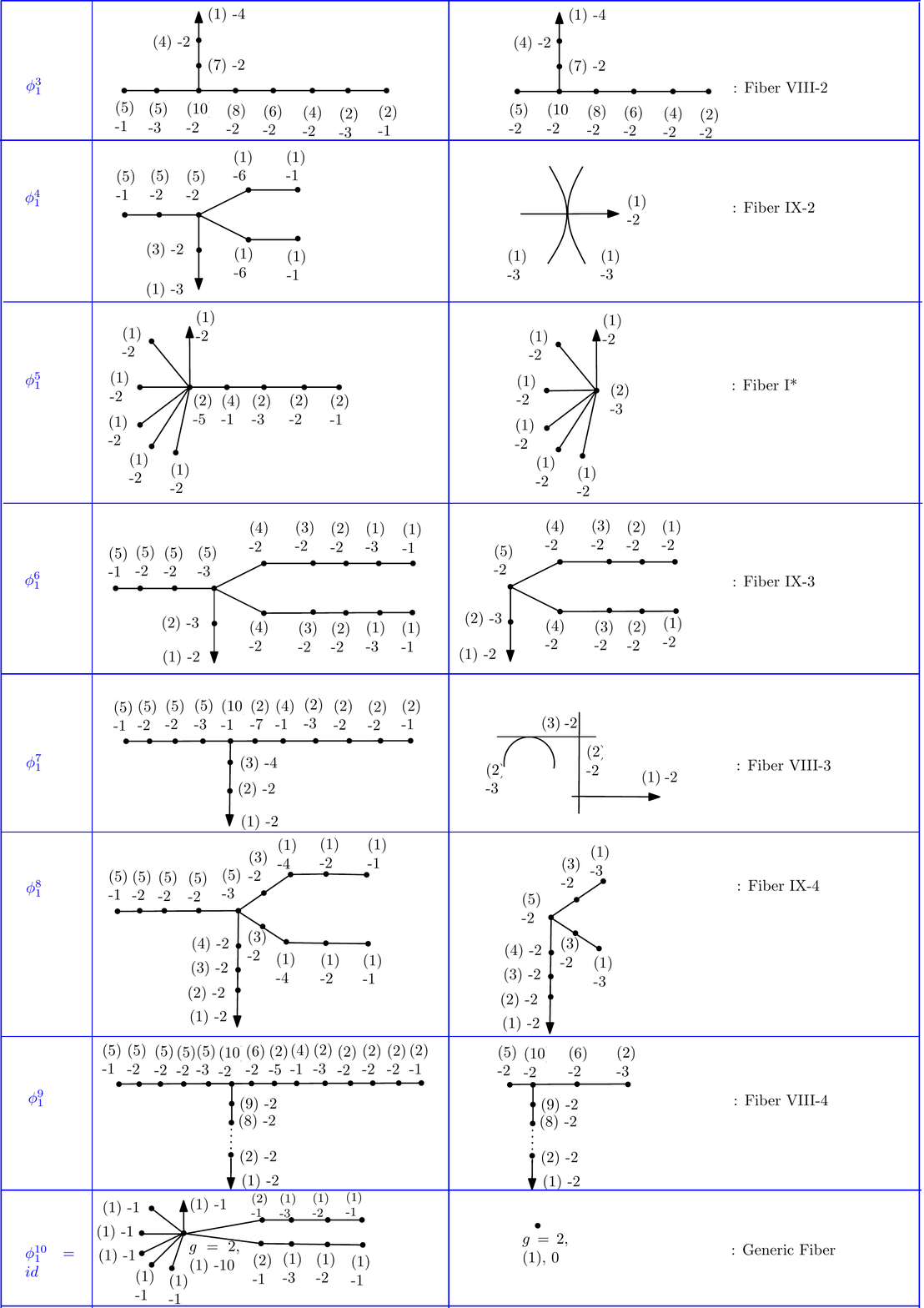}}
\caption{$\phi_1^k$, $k=3, \cdots, 10$}
\label{phi1ktable}
\end{figure}

\subsection{\texorpdfstring{{\boldmath $\phi_3^k$: $y^2 = x^6+t^k$}}{Phi3K}}
\label{phi_3s}
To find the resolution graphs of $\phi_3^k$, we proceed as in the $\phi_1^k$ case above. We first resolve the plane curve singularity $x^6-y^2 =0$ by successive blow-ups and obtain a tree $\mathcal{D}$ with two arrows and three vertices with multiplicities 6, 4, 2 and self intersections $-1,-2,-2$, respectively, as shown in the first step of Figure~\ref{phi3}.

\noindent {\boldmath $k=1$}:  
To find the resolution of $y^2 = x^6+t$, we apply the Nemethi's algorithm with $N=1$ to the tree $\mathcal{D}$. The degree of the cover of $\mathcal{D}$ is $N=1$, so we just blow down the $-1$ curves of $\mathcal{D}$ consecutively. As a result we obtain two 2-spheres of self intersections $-3$ and multiplicities $1$, intersecting each other once with intersection multiplicity 3 (see Figure~\ref{phi3}). Thus we obtain the fiber of type V in the list of Namikawa-Ueno in \cite{NamikawaUeno-list}.
\begin{figure}[htb]
%\centering
{\includegraphics[width=15cm]{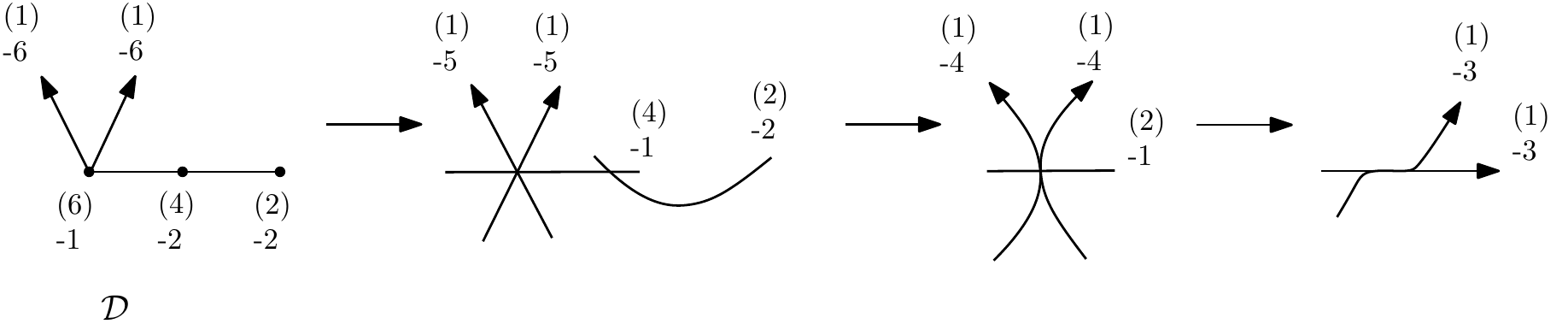}}
\caption{$\phi_3$ - Fiber V}
\label{phi3}
\end{figure}

\noindent {\boldmath $k=2, \cdots, 6$}:  
We find the resolution graphs of $y^2 = x^6+t^k$, $k=2, \cdots, 6$ by applying the Nemethi's algorithm with $N=k$, to the tree $\mathcal{D}$. We summarize them in Figure~\ref{phi3ktable} where the second column consists of the graphs after applying Nemethi's algorithm, and the third column shows the resulting graphs after consecutive blow-downs. Note that in the first case, $\phi_3^2$, after applying Nemethi's algorithm we directly get a minimal graph and it is fiber of type III. Also note that the resulting minimal resolution graphs of $\phi_3^2$ and $\phi_3^4$ both give the type III fiber.\\
\begin{figure}[htb]
\centering
{\includegraphics[width=15cm]{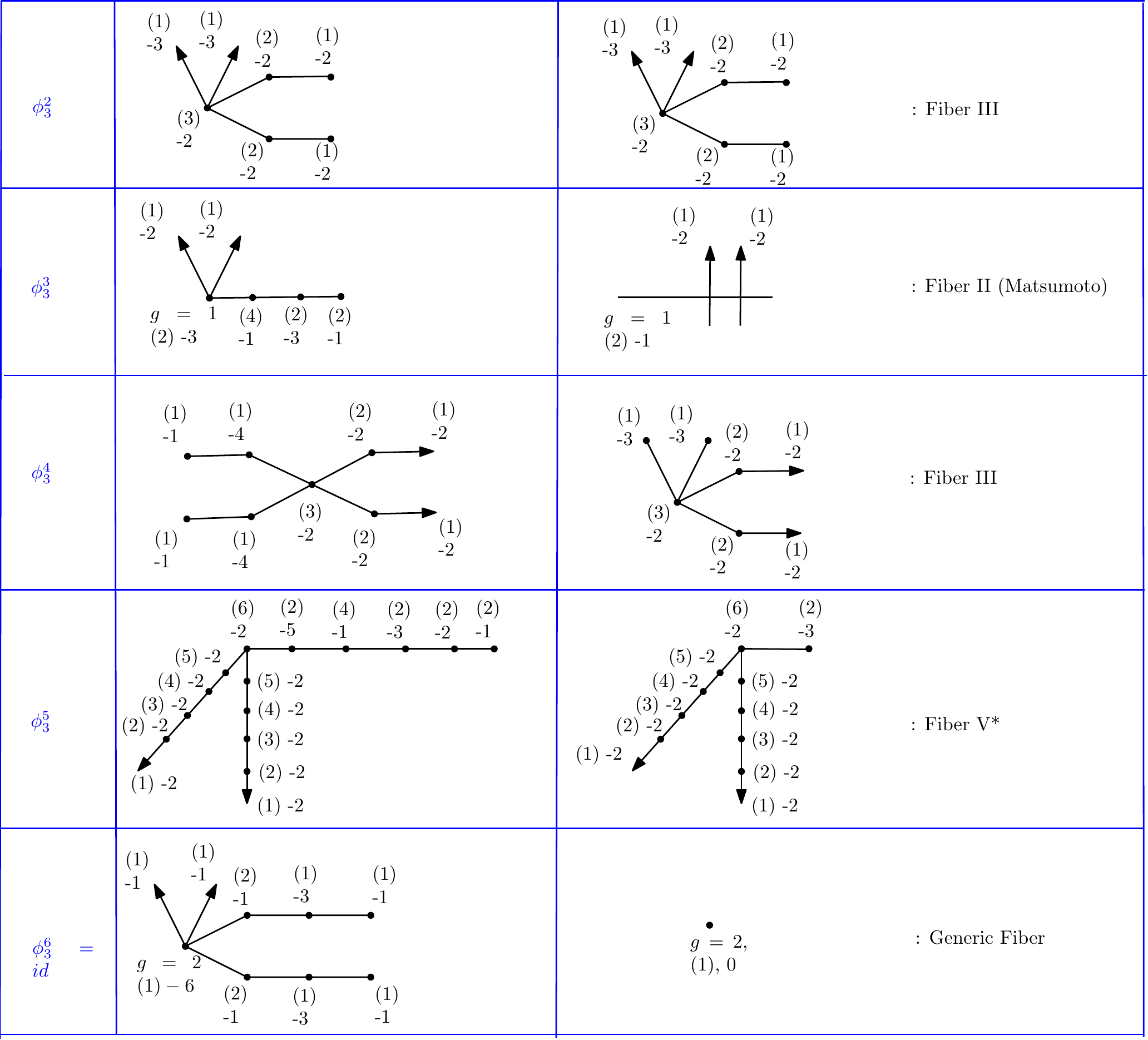}}
\caption{$\phi_3^k$, $k=2, \cdots, 6$}
\label{phi3ktable}
\end{figure}

\noindent {\bf Taking degree two covers}

In the previous subsections we started with the plane curve singularities on the $xy$-plane and resolved them by successive blow-ups. In the resulting graphs, the arrows are the proper transforms of the starting singular curves. Then to those graphs we applied Nemethi's algorithm (\cite{Nemethi-algorithm}, Section 3.6 and Theorem 3.7 a)) with different values of $N$ in each case (where $N$ is the power of $t$ in the polynomials); i.e., we took different degree covers of the graphs in each case. Hence we constructed 13 types of the Namikawa-Ueno fibers as listed in the figures~\ref{phi1}, ~\ref{phi12} ~\ref{phi1ktable}, ~\ref{phi3}, ~\ref{phi3ktable}. 

However, to find the resolution graphs in the following subsections we start with the plane curve singularities on the \emph{$xt$-plane} and resolve them by consecutive blow-ups. Then to the resolution graphs we apply Nemethi's algorithm with \emph{$N=2$} (which is the power of $y$ in all of the polynomials); i.e., we take degree 2 covers of the graphs in all of the following cases. Also in the remaining subsections we drop the arrows and the multiplicities of the irreducible components of the resolution graphs. For this version of the algorithm, see \cite{Nemethi-algorithm}, Theorem 3.7 b).

\subsection{\texorpdfstring{{\boldmath $\phi_2^k: y^2 = x(x^4+t^k)$}}{Phi2K}}
%\subsection{$\phi_2^k: y^2 = x(x^4+t^k)$}
\label{sec:phi2kres}

\noindent {\boldmath $k=1$}: We first resolve $x(x^4+t)$ on the $xt$-plane by iterated blow-ups which gives us the first graph in Figure~\ref{phi2}. Then, to this graph we apply Nemethi's algorithm with $N=2$ to resolve $y^2 = x(x^4+t)$. The second graph in Figure~\ref{phi2} shows the resulting configuration. Next, we drop the arrows as shown in the third step of the figure. Finally, we blow down the $-1$ curves successively and obtain a 2-sphere of square $-2$. In the last step we show that this is indeed the subset of type VII fiber of Namikawa-Ueno, which is given by the same polynomial $y^2 = x(x^4+t)$. We show the missing component, the $(2,3)$-cusp, in blue in the picture. 
\begin{figure}[htb]
%\centering
{\includegraphics[width=15cm]{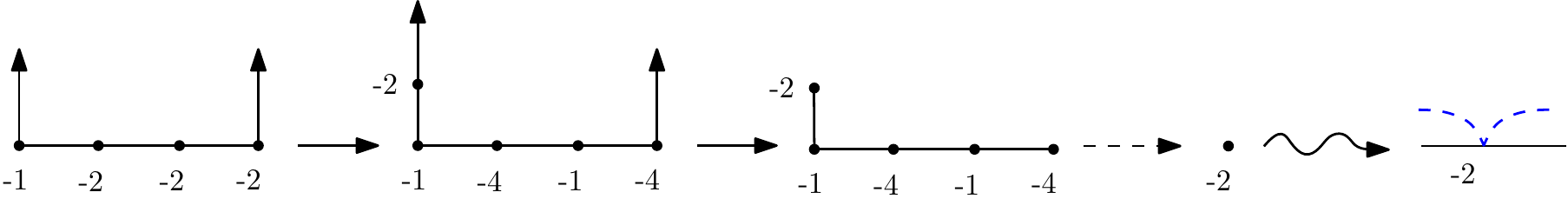}}
\caption{$\phi_2$ and Fiber VII}
\label{phi2}
\end{figure}

\noindent {\boldmath $k=2, \cdots, 8$}: To find the resolutions of $y^2 = x(x^4+t^k)$ for $k=2, \cdots, 8$, we proceed as in the $k=1$ case above. We summarize them in Figure~\ref{phi2ktable}. The first column shows the resolution graphs of $x(x^4+t^k)$ on the $xt$-plane. Then we take degree two covers of these graphs, i.e., we apply the Nemethi's algorithm with $N=2$, to the graphs of $x(x^4+t^k)$. The second column shows the resulting resolution graphs of $y^2 = x(x^4+t^k)$ after applying the algorithm. Next, we drop the arrows and blow down all $-1$ spheres. We show these graphs in the last column of Figure~\ref{phi2ktable} in black. We note that they are subsets of certain types of the Namikawa-Ueno fibers. In the last column, we note the types of the Namikawa-Ueno fibers that these graphs are subsets of. The missing components are depicted in blue.
\begin{figure}[htb]
%\centering
{\includegraphics[width=15cm]{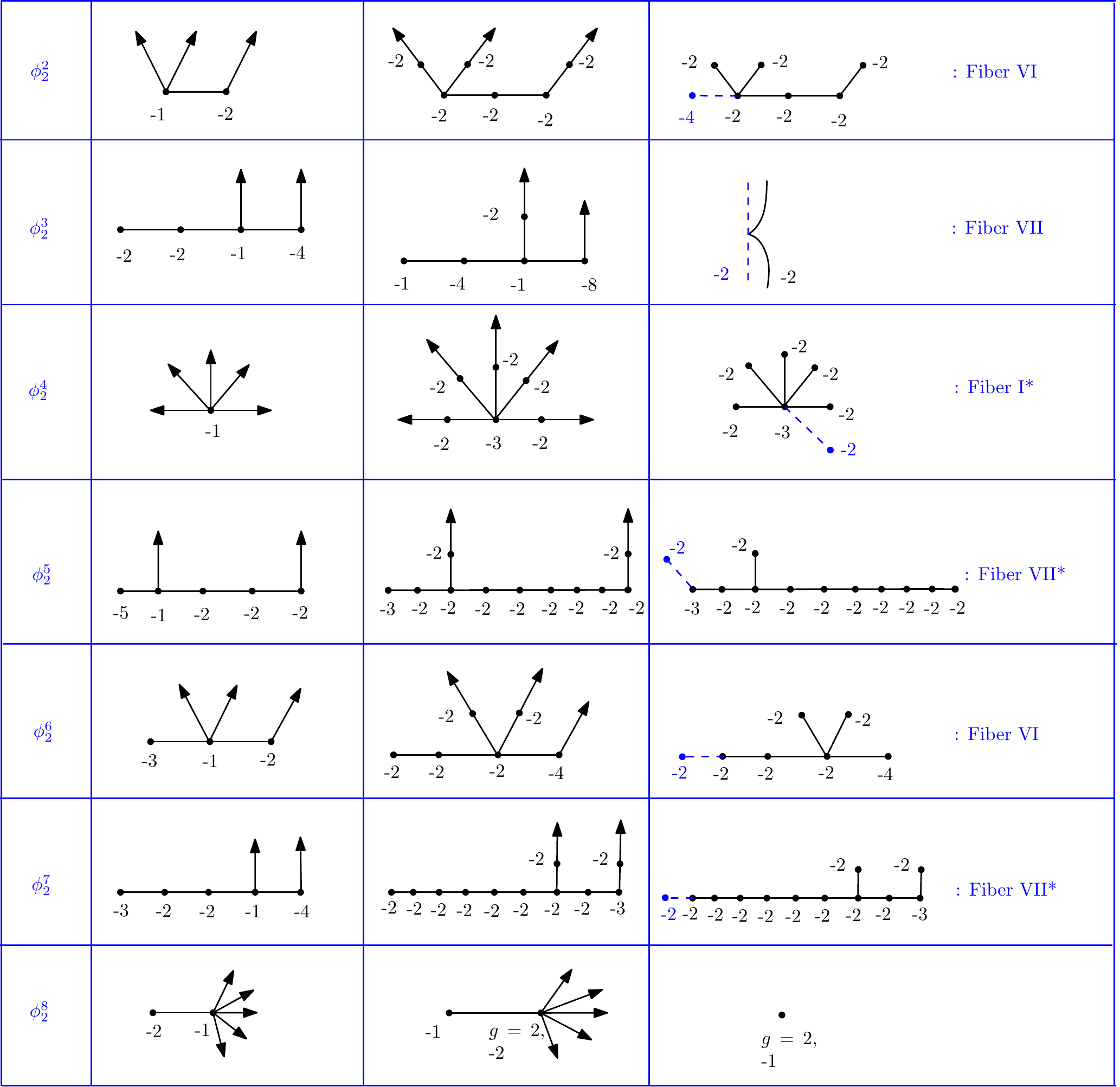}}
\caption{$\phi_2^k$, $k=2, \cdots, 8$}
\label{phi2ktable}
\end{figure}

\subsection{\texorpdfstring{{\boldmath $\phi_4^k: y^2 = x(x^5+t^k)$, $k=1, \cdots, 5$}}{Phi4K}}
%\subsection{$\phi_4^k: y^2 = x(x^5+t^k)$} $k=1, \cdots, 5$: 
To find the resolutions of $y^2 = x(x^5+t^k)$ for $k=1, \cdots, 5$, we proceed as in the $\phi_2^k$ case above. We summarize them in Figure~\ref{phi4ktable}. The first column shows the resolution graphs of $x(x^5+t^k)$ on the $xt$-plane. Then we take degree two covers of these graphs, i.e., we apply the Nemethi's algorithm with $N=2$, to the graphs of $x(x^5+t^k)$. The second column shows the resulting resolution graphs of $y^2 = x(x^5+t^k)$ after applying the algorithm. Next, we drop the arrows and blow down all $-1$ spheres. We show these graphs in the last column of Figure~\ref{phi4ktable}.
\begin{figure}[htb]
%\centering
{\includegraphics[width=15cm]{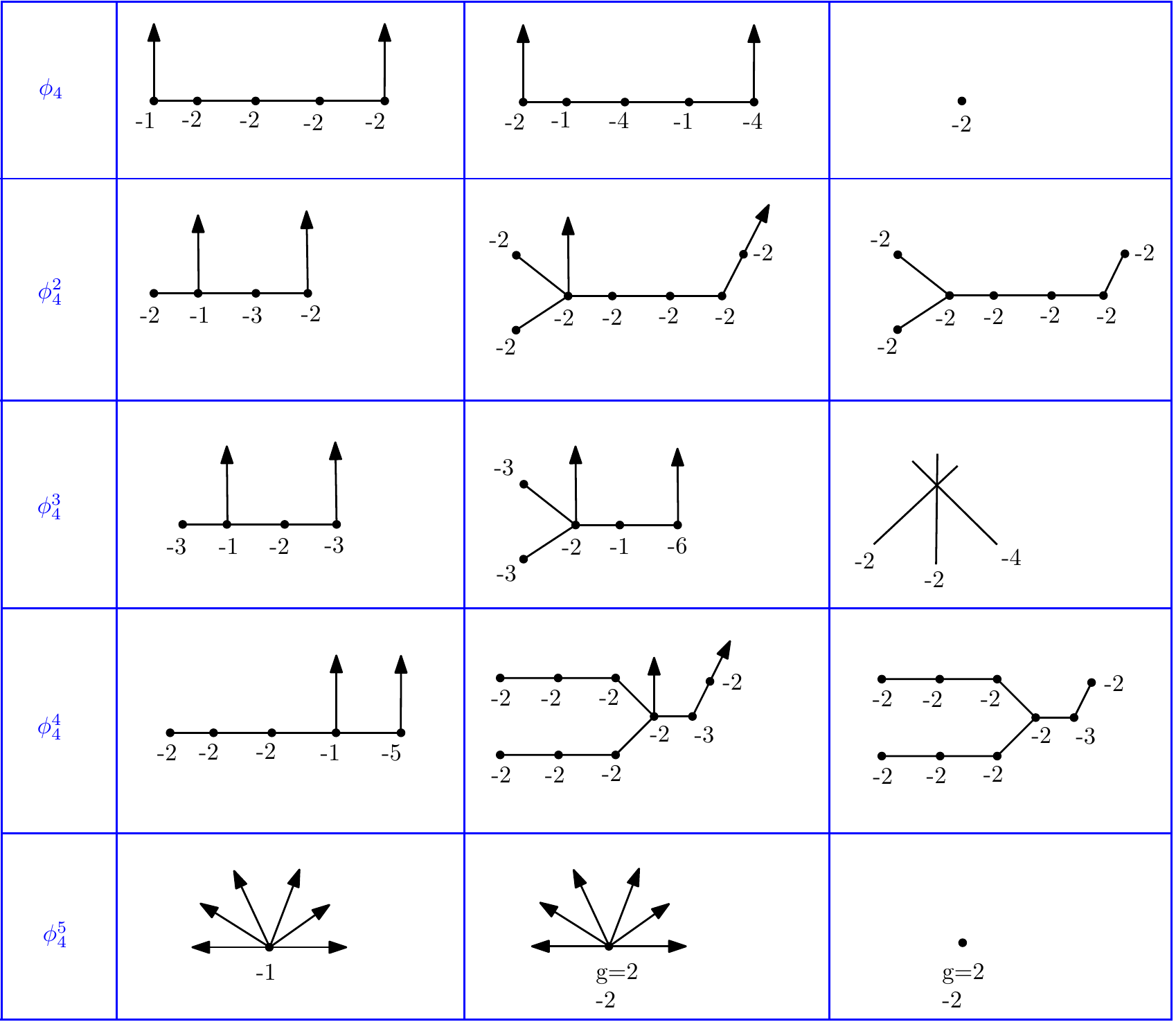}}
\caption{$\phi_4^k$, $k=1, \cdots, 5$}
\label{phi4ktable}
\end{figure}

\subsection{Resolutions of lower genus boundary twists} In this last section we consider 4 special cases. To find the resolutions we follow the steps as in $\phi_2^k$ and $\phi_4^k$ cases above. Namely, we first resolve the plane curve singularities on the $xt$-plane. Then we apply the Nemethi's algorithm with $N=2$. Next, we drop the arrows and blow down all $-1$ spheres.

\noindent {\boldmath $y^2= x^3 +t^6$}: We first resolve $x^3 +t^6=0$ and obtain the first graph in Figure \ref{36}. After taking the degree 2 cover via Nemethi's algorithm, we get the second graph in the same figure. Then we drop the arrows and blow-down the $-1$ sphere. Hence we get a vertex of genus one and self intersection $-1$.
\begin{figure}[htb]
%\centering
{\includegraphics[width=12cm]{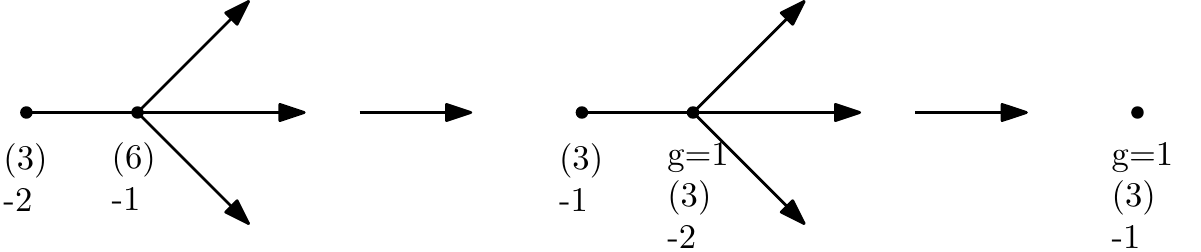}}
\caption{Resolutions of $y^2= x^3 +t^6$ and $y^2= x(x^2 +t^4)$}
\label{36}
\end{figure}

\noindent {\boldmath $y^2= x(x^2 +t^4)$}: This case is identical to the previous case. When we resolve $x(x^2 +t^4)$, we again obtain the first graph of Figure \ref{36}. Therefore, after applying the same steps we get a vertex of genus one and self intersection $-1$.

\noindent {\boldmath $y^2= x^4 +t^4$}: \label{res:bdry44} We first resolve $x^4 +t^4=0$ and obtain the first graph in Figure \ref{44}. After taking the degree 2 cover via Nemethi's algorithm, we get the second graph in the same figure. Then we drop the arrows and hence, we get a vertex of genus one and self intersection $-2$.
\begin{figure}[htb]
%\centering
{\includegraphics[width=9cm]{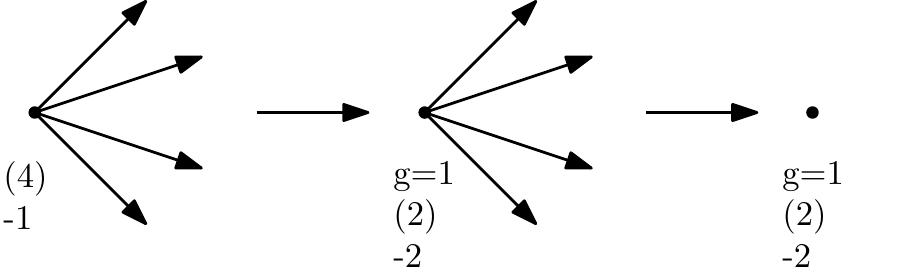}}
\caption{Resolutions of $y^2= x^4 +t^4$ and $y^2= x(x^3 +t^3)$}
\label{44}
\end{figure}

\noindent {\boldmath $y^2= x(x^3 +t^3)$:} Once again, this case is identical to the previous case. When we resolve $x(x^3 +t^3)$, we obtain the first graph of Figure \ref{44}, as in the resolution of $y^2= x^4 +t^4$. Therefore, at the end we again obtain a vertex of genus one and self intersection $-2$.

\clearpage

\section{First examples: \texorpdfstring{$\phi_1$}{ϕ1} and \texorpdfstring{$\phi_1^2$}{ϕ1\textasciicircum 2}}\label{sec:phi1and2}  

\subsection{\texorpdfstring{{\boldmath $\phi_1$}}{Phi1}}\label{sec:phi1} 
We begin this section with a quick outline of how we construct the deformation from the fibration with a singular fiber to a Lefschetz fibration. We refer to this process as \emph{splitting}. $X$ will be the resolution of some affine variety $V$ defined by a polynomial in $\CC^3$. We will always use polynomials of the form $y^2 - f(x,t)$ so we write $V(f)$ for this variety and $X_f$ for its resolution. Note that by construction, $V(f)$ always admits the hyperelliptic involution given by $(x,y,t)\to (x,-y,t)$ and that structure will be essential in characterizing our splitting. Additionally, $V(f)$ will have the structure of a fibration $\pi_t:V(f) \to \CC$ by projecting onto the $t$ coordinate and for the polynomials we consider, all non-singular fibers will be smooth curves of genus 2 with one or two boundary components. The fiber over $t=0$ will be singular with an isolated singularity at the origin which is always more degenerate than a standard Lefschetz singularity. Usually $V(f)$ will itself be a singular algebraic surface, in which case we will resolve $V(f)$ to $X_f$. The fibration $\pi_t$ lifts to $X_f$ and the fiber over $0$ carries the plumbing tree of the resolution. We want to deform $(X_f, \pi_t)$ through symplectic fibrations to a Lefschetz fibration $(X_s, \tilde{\pi})$ and record the vanishing cycles along with the identifications of the reference fibers of the two fibrations, that is, we want to find a \emph{splitting} of the singular fibration $\pi_t$ on $X_f$.

As a first example, we consider the case $y^2 - (x^5 + t)$. Checking the derivatives, $\p /\p t = 1 \neq 0$ so $V(f)$ is already a smooth subvariety of $\CC^3$ and $X=X_f = V(f)$. $X$ admits a fibration $\pi_t$ (by intersecting with the hyperplanes $t = const$). This fibration is smooth except over $t=0$. Fibers are given as the hypersurface quintics $y^2-x^5 = -c$ in $\CC^2$ identified with (and subsets of) the hyperplane $t=c$. For $t\neq 0$, these are all smooth genus 2 surfaces in $\CC^2$ and the link with $S^3s$ of large radius is a $(2,5)$-torus knot. At $t=0$ the fiber is a cuspidal quintic, $y^2 = x^5$. This is the singularity associated to the $(2,5)$-torus knot. Indeed, $X = \CC^2$ and the fibration $\pi_t$ is the Milnor fibration associated to the $(2,5)$-torus knot: $\pi(x,y) = x^5-y^2$.

Associated to the polynomial $f$, we have a smooth algebraic variety $X$ and a fibration by genus 2 surfaces (with one boundary component) which is smooth except for the central fiber $t=0$. We want to deform this fibration structure (and possibly the algebraic structure on $X$) to one in which all singular fibers are of Lefschetz type. To do this, we make use of the hyperelliptic nature of $f$: $X$ admits a hyperelliptic map $h$ (so $\pi_t \circ h = \pi_t$). This map is given explicitly as $y\mapsto-y$. The $h$ action is not free but the quotient is $X/h = \CC^2$ and the branch locus is given by $\Delta = \{y=0\} = \{0 = x^5 + t\}$. This is a smooth subvariety of $\CC^2$, topologically a disk bounded by the $(5,1)$-torus knot, i.e., the unknot. The double branched cover of $\CC^2$ over $\Delta$ is again $\CC^2$. From this perspective, the fibration is singular at $t=0$ because $\Delta$ is tangent to the hyperplane $t=0$ (to order 5).

To split the singularity at $t=0$, we algebraically deform $\Delta$: $\Delta_s = \{ x^5 + sx + t = 0\}$ and lift the corresponding deformation to $V$ and $X$. Each $\Delta_s$ is still smooth and is isotopic to $\Delta$ through smooth subvarieties of $\CC^2$. The corresponding family $X_s$ built by taking the double cover of $\Delta_s$ consists of deformation equivalent copies of $\CC^2$ but where the fibration $\pi_t$ changes. For $s\neq 0$, $\Delta_s$ has 4 points of tangency with the $t$ fibers which occur at the points $$(t,x) = \left( \frac{-4s}{5} \left(\frac{-s}{5}\right)^{1/4}, \left(\frac{-s}{5}\right)^{1/4} \right)$$

Each of these points $(c,p)$ is tangent of order 2: $\Delta_s$ is given by $ \{(t-c) = C(x-p)^2(x-p_1)(x-p_2)(x-p_3)\}$ for some nonzero value of $C$ and distinct values of $p,p_1,p_2,p_3$. On the double branched cover $X_s$ we have a Lefschetz singularity near the lift of $(c,p)$ (with local model $y^2 - x^2 = 0$ in suitable coordinates). Looking at the singular values of the fibration in the $t$-plane, as you increase $s$ from $0$ to $1$, the singular fiber at $t=0$ splits into 4 different places of simple branching lying along the rays with polar angles $\pm \pi/4$ and $\pm 3\pi/4$ and having radius equal to $4 \frac{|s|^{5/4}}{5^{5/4}}$. All the fibrations for $s\neq 0$ are isomorphic Lefschetz fibrations. We want to give a full description of $\pi^s_t$ on $X_s$. For us, that means a description of the Lefschetz fibration on $X_s$, usually as an ordered list of the \emph{vanishing cycles} of $\pi^s_t$ along with an identification of the reference fiber $(\pi^s_t)^{-1}(1)$ with the reference fiber $\pi_t^{-1}(1)$ of $X$. 

To understand the branching of $\Delta_s$ completely, we the traverse the circles of radii close to $4 \frac{|s|^{5/4}}{5^{5/4}}$ in the $t$-plane and watch the braiding and collisions that happen in $\Delta_s$ (that is, in the intersection of $\Delta_s$ with $(\pi^s_t)^{-1}(p)$). At $s=0$, $\Delta$ over $p = 1$ consists of the 5th roots of $-1$. As $s$ increases, these deform slightly, with the real root moving from $-1$ to $-4/5^{5/4}$. Fixing a non-zero value of $s$, we will traverse the path in the $t$-plane from $1$ to $4 \frac{|s|^{5/4}}{5^{5/4}}$ along the real line, followed by a counterclockwise traversal of the circle of radius $4 \frac{|s|^{5/4}}{5^{5/4}}$, and then finally back to $1$, again along the real line. Moving in and out along the real line just slides the intersection point at -1 in and out along the real axis. As we traverse the circle, the point along the real axis follows a counter clockwise path, colliding and braiding with the 4 other points in order. Label the moving point 1 and the rest in counter clockwise order are 2, 3, 4, 5. Pulling these collisions back to the reference point at $t=1$, we see that there are four bands being added. The first is between 1 and 2 corresponding to a disk band and the next has the moving point continuing counterclockwise to before braiding with 3 and continuing on to 4 and then finally 5. If we want to write this factorization in the braid group, though, we need fixed reference points for the braiding and we need to choose anchoring arcs in the $t$-plane connecting $t=1$ to each of the four points where branching occurs. To match this with the collisions we see above, we take counterclockwise arcs of fixed radius starting at $1$, stopping at the ray through the corresponding branch point and then traveling inward along the ray until we hit the branch point. We assume $s$ is a small real number so that the branch points occur at points in $\CC$ of small radius. As we approach the branch point we see two of the marked points collide, giving the local braiding and a band for the surface $\Delta_s$. To complete this to a quasipositive factorization, we mark the arc that indicates the paths of the two points as they collide and then we pull that arc back to $t=1$ via the path above. That arc gives a quasipositive half twist in the braid description of the $(5,1)$ torus knot and the quasipostive band description of the perturbed branch surface $\Delta_s$.

The paths in the $t$-plane described above are shown in Figure~\ref{fig:phi11} along with the corresponding bands pulled back to the reference fiber at $t=1$. Under the deformation of $s$, the fibration $X$ splits into a Lefschetz fibration associated to the positive braid with vanishing cycles coming from the braid twists $a_1 a_2 a_3 a_4$ (which, as a braid, is written left to right). The corresponding Lefschetz fibration has a word in Dehn twists as $$\phi_1 = \tau_1 \tau_2 \tau_3 \tau_4$$ (again written left to right) where the $\tau_i$ are the Dehn twists in the standard chain, identified with the fiber above $t=1$ using the skeleton shown in Figure~\ref{fig:phi11}.

\begin{figure}
\includegraphics[width=5.5in]{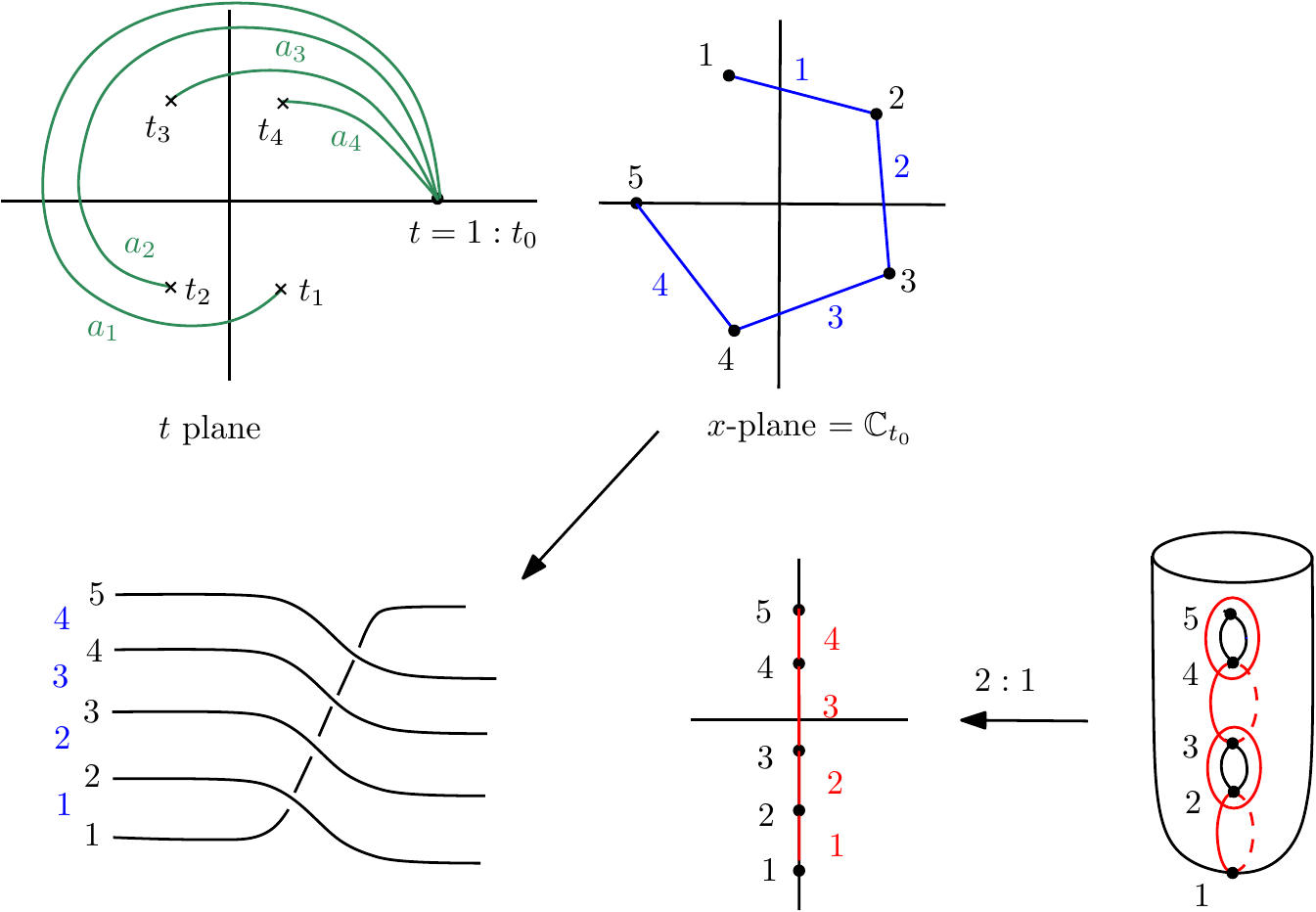} 
\caption{The quasipositive braid factorization $(q_i)$ associated to $\Delta_s$ for $f(x,t) = x^5 + t$. Top: On the left we see the paths chosen in the $t$-plane and on the right the corresponding arcs $q_i$ along which the braiding occurs. Below that we see the choice of the double branched cover (indicated by a spine) and the corresponding factorization in the mapping class group of the genus 2 surface. At the bottom left corner we see the corresponding quasipositive factorization of the braid formed by intersecting $\Delta$ with the boundary of the polydisk in $\mathbb{C}^2$. The half-twists are labeled by numbers to indicate the order they occur in the braid word. As noted in Lemma~\ref{lm:phi1rotate}, this factorization is Hurwitz equivalent to each of those obtained by rotating the given factorization about the origin through an angle of any multiple of $2\pi/5$.}
\label{fig:phi11}
\end{figure}

\begin{lemma} \label{lm:hurwitz} Given a braid factorization $B = B_f ab B_e$ with quasipositive bands $a$ and $b$ as shown in Figure~\ref{fig:braidhurwitz}, $B$ is Hurwitz (or band) equivalent to both $B_f b c B_e$ and $B_f c a B_e$, where the local picture of $a$, $b$ and $c = b(a) = a^{-1}(b)$  is as shown in the figure. The same holds for the Dehn twist factorization in the double branched cover.
\end{lemma}

\begin{lemma} \label{lm:phi1rotate} The factorization $\phi_1 = \tau_1 \tau_2 \tau_3 \tau_4$ is Hurwitz equivalent to all other factorizations of $\phi_1$ given by rotating the braid factorization in Figure~\ref{fig:phi11} by some multiple of $2\pi/5$ and lifting via the branched double cover. More explicitly, since the monodromy $\phi_1$ is represented by this rigid rotation by $2\pi/5$, we have that $\phi_1^k (\phi_1)$ is Hurwitz equivalent to $\phi_1$.
\end{lemma}

Note: Here we really mean Hurwitz equivalent in the strongest sense: only Hurwitz moves are allowed or needed and neither cyclic permutations nor global conjugation is used in the proof of the theorem. This is particularly important when we use this word to  compose/decompose factorizations of compound singularities later on. 

\begin{proof} Looking at this from the perspective of a solitary singularity, this is obvious. We can always conjugate a monodromy by itself by taking the paths used to identify the vanishing cycles and concatenate all with a full counterclockwise loop enclosing all the singular values. However, because it's not complicated, we include a separate, direct proof. Repeatedly applying Lemma~\ref{lm:hurwitz} we can slide the first Dehn twist in the factorization of $\phi_1$ around to the empty edge of the pentagon. (And similarly we could also use Hurwitz moves to slide the last Dehn twist around to the empty edge.)
\end{proof}

\begin{figure}  
\begin{center}
\includegraphics[width=4.5in]{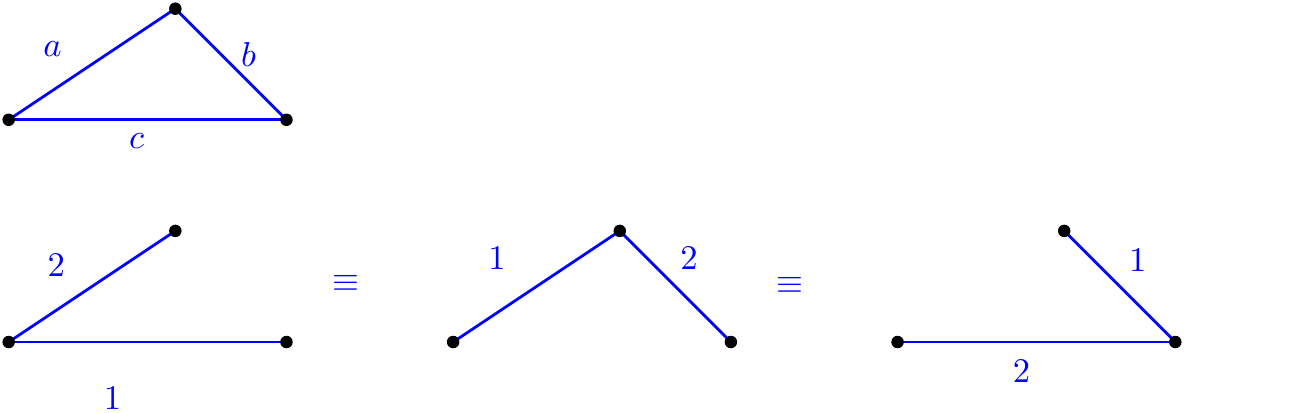}
\caption{The local picture of a braid factorization and two band/Hurwitz equivalent factorizations. The top triangle labels the configuration of arcs in the plane. In each of the diagrams along the bottom row, the blue arcs represent arcs along which we do braid half-twists and the numbering indicates the order in which the braid half-twists occur. So along the bottom, from left to right we have the local subwords of $ca= a(b) a$, $a b$, and $b b(a) = bc$. The content of Lemma~\ref{lm:hurwitz} is that all three of these factorizations are Hurwitz equivalent.}
\label{fig:braidhurwitz}
\end{center}
\end{figure}

%%%%%%%%%%%%%%
\subsection{\texorpdfstring{{\boldmath $\phi_1^2$}}{ϕ1\textasciicircum2}} \label{sec:phi12}
There are only a few polynomials that we study where the hypersurface in $\CC^3$ is smooth. To illustrate how we decompose more complicated singularities, we start with the representative case of  $y^2 - (x^5 + t^2)$.

We will show how the fibration associated to the resolution of $y^2 - (x^5 + t^2)$ splits into two copies of $\phi_1$ so that the resolution $X_f$ admits the Lefschetz fibration associated to the word $\phi_1^2 = (\tau_1 \tau_2 \tau_3 \tau_4)^2$ where $\phi_1$ is the factorization found in Section \ref{sec:phi1}. We need to argue that this is actually a deformation of the resolution, so that $X$ is deformation equivalent to the Lefschetz fibration given by the positive word $\phi_1^2 = (\tau_1 \tau_2 \tau_3 \tau_4)^2$. This will be accomplished by checking Laufer's criterion for weak simultaneous resolutions, i.e., that $K^2$ is constant. First we point out some observations:

\begin{itemize}
    \item The singular variety $V(f)$ is the standard $A_4$ Du Val singularity and so its resolution $X_f = X$ is a plumbing of $-2$ spheres plumbed along a type $A_4$ Dynkin diagram. This manifold is negative definite with $b_2 = 4$.
    \item The deformation $f_s = y^2 - (x^5 + t(t-s))$ splits the singularity of $f=f_0$ into two singularities, each isomorphic to $y^2 = x^5 + t$ (i.e., of type $\phi_1$), with one at $(t,x,y) = (0,0,0)$ and one at $(s,0,0)$. 
    \item The specific local models for the two singularities are $y^2 = x^5 - t$ at $t=0$ and $y^2 = x^5 + t$ at $t=s$, so this yields a smooth surface $X_s$. The latter singularity can be identified with the factorization $\phi_1$ used for $x^5 + t$ just as in the previous case, but the former is different. We identify with the model for $x^5 + t$ the original configuration by $t \rightarrow -t$, and measuring relative to the reference point $t = -1$ along an arc along the negative real axis. If we choose a path from $t=1$ to $t=-1$ we can pull this factorization back to our original basepoint at $t=1$. If we take that path to lie on the unit circle and with negative imaginary coordinate, then the identification of the fiber at $t=1$ with the fiber at $t = -1$ is by a clockwise rotation of the fiber through $2\pi/5$, and so pulling the factorization back, we get the Lefschetz fibration associated to the factorization of the monodromy as $\phi_1 \phi_1(\phi_1)$. However, as was demonstrated in Lemma~\ref{lm:phi1rotate}, this factorization is Hurwitz equivalent to the easier, $\phi_1^2$.
    
    \item Each of $X_s$ is smooth, and is worth noting that (redundant to later observations) a checking of the handle decomposition of $X_s$ coming from the Lefschetz fibration $\phi_1^2$, each $X_s$ is diffeomorphic to a plumbing of $-2$ spheres along the $A_4$ Dynkin diagram. 

    \item Each $X_s$ is a Lefschetz fibration with boundary an open book on $Y = \bdry X_s$, independent of $s$.
    \item Each $X_s$ is pseudo convex with boundary a plane field isotopic to the contact structure $\xi$ carried by this open book.
    \item For each $X_s$ we have $$K(X_s)^2=c_1^2(X_s) = 4 d_3(\xi)+2\chi(X_s) + 3 \sigma(X_s)$$
    where $\chi(X_s)$ is the Euler characteristic and $\sigma(X_s)$ is the signature, and $d_3$ is Gompf's 3-dimensional invariant of the plane field $\xi$ \cite{Gompf}.
    \item As $X_s$ and $X_0$ are diffeomorphic, they have the same Euler characteristic and signature, so $K(X_s)^2$ is constant.
\end{itemize}

So the final observation is that both all of $X_0$ and the nearby fibers $X_s$ have the same values of $K^2$ so Laufer's theorem applies and (after a possible finite base change) the family is flat. Hence the deformation above by $X_s$ for $s$ a positive real number represents a deformation of the resolution $X_f = X_0$ and all nearby fibers $X_s$ are diffeomorphic to $X_f$ and represent a deformation of the complex structure on $X_f$.

%%%%%%%%%%%%%%%%
\section{Smaller degree base cases: Lower genus boundary twists}
\label{sec:lower genus}
While the following cases are amenable to the same kinds of arguments we give for the remaining cases discussed later, we can address them directly using easier techniques and we will additionally need these results in the later examples.

\noindent {\boldmath ${y^2= x^4 +t}$}: 
\noindent This case is the direct analogue of our introductory example of $\phi_1$ above. After deforming, this yields a the factorization $\tau_1 \tau_2 \tau_3$ of the 3 chain in the genus 1 surface with 2 boundary components. 

\noindent {\boldmath ${y^2= x^4 +t^4}$}:\label{ex:bdry44}
\noindent This case is similar. Its hyperelliptic quotient is the cone on the $(4,4)$-torus knot. The singular surface has a fibration by twice punctured tori and the monodromy is a single boundary multitwist. The resolution has $b_1 = 2$ and $b_2 = 1$ and is represented by a complex curve of genus 1. The monodromy factorization must be by a single pair of twists about homologous, essential curves whose product is the boundary multitwist. Capping the boundary components, such factorization yields a (non-minimal) closed elliptic fibration with two singular fibers, so the two Dehn twists in the factorization must be about curves which are nullhomotopic in the closed genus 1 surface. In the twice punctured torus, they then must be both homologically essential and separating and so must be boundary parallel. This gives the expected factorization into the multitwist about the two boundary components.

\noindent {\boldmath ${y^2= x(x^3 +t^3)}$}:
This is hyperelliptic with hyperelliptic quotient the cone on a union of the unknot and its $(3,3)$-torus knot satellite, braided as such in $\CC^2$, with the fibration structure given by projection onto the first factor. The double cover is again a singular surface fibered by twice-punctured tori. For the exact reasons above, the monodromy factorization must be by a pair of Dehn twists about curves that separate off a genus 1 surface and so must be the boundary multitwist.

\section{Splittings}
\label{sec:split}
In this section we will complete the proof of Theorem~\ref{thm:main2} by constructing splittings for all the remaining cases.

\subsection{\texorpdfstring{\boldmath $\phi_1^k$}{phi1k}} We start with the family $y^2 = x^5+t^k$.

To make our lives easier overall in this section, we will change our identification of the branched double cover of $\CC$ branched over the points $x^5+1 = 0$ with the surface $\Sigma_{2,1}$ just slightly. In Figure~\ref{fig:phi1skeleton}, we give the graph and ordering that will lift to the standard generators of the hyperelliptic mapping class group as shown in Figure~\ref{fig:gens}. This is related to the identification used in the case $y^2 = x^5+t$ from Section~\ref{sec:phi1} by a rotation in the plane.
\smskip

\begin{figure}
    \centering
    \includegraphics[width = 1.6in]{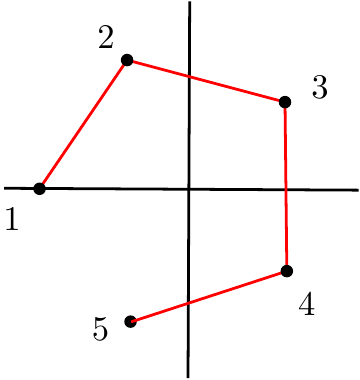}
    \caption{The skeleton that identifies the curve $y^2 = x^5+1$ over $t=1$ with the standard genus 2 surface $\Sigma_{2,1}$.}
    \label{fig:phi1skeleton}
\end{figure}
\begin{figure}
    \centering
    \includegraphics[width = 5in]{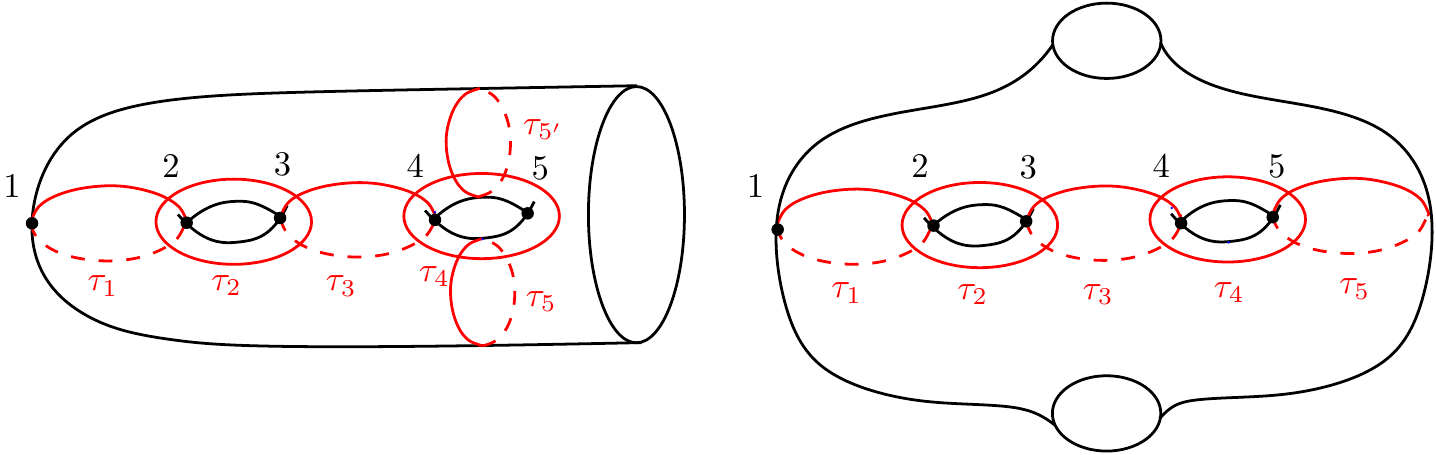}
    \caption{The arrangement of the positive Dehn twists used in the Lefschetz fibrations.}
    \label{fig:gens}
\end{figure}

\noindent \textbf{Case} {\boldmath $y^2= x^5 +t$}:
We saw in Section~\ref{sec:phi1} that $X$ deforms into the Lefschetz fibration with monodromy given by the word $\phi_1 = \tau_1 \tau_2 \tau_3 \tau_4$, but in the skeleton used in that section. To achieve this factorization for this skeleton, we just invoke Lemma~\ref{lm:phi1rotate}. 
\smskip

\noindent {\boldmath  \textbf{Case} $y^2= x^5 +t^2$}:
We saw in Section~\ref{sec:phi12} that $X$ deforms into the Lefschetz fibration with monodromy given by the word $\phi_1^2 = \left( \tau_1 \tau_2 \tau_3 \tau_4 \right)^2$. 
\smskip

\noindent {\boldmath  \textbf{Case} $y^2= x^5 +t^3$}:
First, observe that since via $t \mapsto t^3$, $V(f)$ is a 3-fold cover of $V(y^2 - (x^5 +t))$ branched over the fiber $t=0$, so the monodromy of the associated open book is $\phi_1^3$. We'll show that indeed this fibration splits into a Lefschetz fibration associated to the positive word $\left( \tau_1 \tau_2 \tau_3 \tau_4 \right)^3.$ 

Using Laufer, Theorem~\ref{thm:laufer}, we'll see that the deformation $$y^2= x^5 +t^3 + st$$ admits a very weak simultaneous resolution and so gives a splitting of the resolution $X_0$. Thus, for sufficiently small $s$, the resolutions $X_s$ and $X_0$ are isomorphic. For $s\neq 0$, we see the singularity at $t=0$ break into three singularities, each of type $y^2 = x^5 + t$, occurring at $t=0$, and $t= \pm \sqrt{s}$. For $s\neq 0$, $X_s$ is smooth and the three singular fibers each deform into copies of $\phi_1$. We could separately deform each of these into a Lefschetz fibration, and then pulling these factorizations back to the basepoint $t=1$. Choosing whichever paths yields a factorization of the form $\phi_1^{k_1}(\phi_1)\phi_1^{k_2}(\phi_1)\phi_1^{k_3}(\phi_1)$ which, by the corollary, is Hurtwitz equivalent to $\phi_1^3$. Thus the fibration deforms into the Lefschetz fibration given by the positive word $\phi_1^3 = \left( \tau_1 \tau_2 \tau_3 \tau_4 \right)^3.$ 

The resolution, $X_0$, of the original singularity is shown in Section~\ref{phi_1s} to be the negative definite $E_8$ plumbing with $b_2 = 8$ (see Figure~\ref{phi1ktable}). %after we drop the arrow.
To compare, we draw the handlebody decomposition for $X_s$ corresponding to $\phi_1^3$. After cancelling each of the 1-handles with one of $\tau_i$, we get a 2-handlebody which looks similar to the diagram of an elliptic fibration as constructed by Fuller \cite{Fuller}. The intersection form is negative definite with $b_2 = 8$, This shows that $K_X^2$ is constant and hence that this deformation admits a very weak simultaneous resolution and that $X_0$ deforms into the Lefschetz fibration $X_s = \phi_1^3$.

\smskip

\noindent {\bf An interlude on negative definiteness} \label{negdef}

It will be helpful to have a single negative semi-definite Lefschetz fibration set out in advance. All the other factorizations given below are sub-fibrations of this example and we'll be able to conclude that all of our later examples are negative definite. This example fibration $I^2$ is the fibration over $D^2$, with fibers of genus 2 and 1 boundary component, associated to the positive word $$I^2 = (\tau_1 \tau_2 \tau_3 \tau_ 4 \tau_ 5 \tau_{5'} \tau_ 4 \tau_3 \tau_ 2 \tau_1)^2$$
(we will see this word $I$ again in the singularity $y^2 = x^5 +t^5$).

The corresponding closed fibration was studied by Endo and Meyer. The above is their example minus a fiber and section. Their neighborhood is a plumbing which has signature 0. Meyer calculated the signature to be -12. The boundary 3-manifold has $b_1 = 4$. In $I^2$ we can cancel all 1-handles. The remaining 16 2-handles all generate $b_2$, so $b_2^0=4$ and $b_2^-=12$ and the fibration $I^2$ is negative semi-definite. All the remaining factorizations that we will see among the cases $y^2 = x^5 +t^k$ are subwords of $I^2$, and since the boundaries of all the remaining fibrations to be discussed are rational homology spheres (save for the fibrations corresponding to a boundary twist or boundary multitwist), the corresponding Lefschetz fibrations are all negative definite.

\smskip

\noindent \textbf{Case} {\boldmath $y^2 = x^5 +t^4$}:
This seems at first to be very similar to the case $y^2= x^5 +t^3$. The link of $V(f)$ has an open book whose monodromy is $\phi_1^4$. However, the resolution of this singularity has $b_2 = 2$ and is negative definite (see Figure~\ref{phi1ktable}) and the Lefschetz fibration associated the the positive word $\left( \tau_1 \tau_2 \tau_3 \tau_4 \right)^4$ has $b_2 = 12$ and is not negative definite. Something else must be occurring.

To determine the correct Lefschetz fibration, we consider the deformation $y^2 = (x+s)x^4 + t^4$. For $s \neq 0$ this has five singularities. At $t=0$ we have a compound singularity of the form $y^2 = x^4 + t^4$ at $(0,0,0)$ (along with a smooth component approximated by $y^2 = x+s$). As we saw in Section~\ref{ex:bdry44}, the resolution of this singularity deforms into the Lefschetz fibration associated to a pair of Dehn twists that splits off a genus 1 surface with two boundary components. There are four other singularities, all of Lefschetz type. To see what they are, it's again easier to look at the branch locus in the quotient of the hyperelliptic involution: $\Delta_s :\{(x+s)x^4 + t^4 = 0\}$. Here, $\Delta_s$ is singular at $t=0$ and the braid we see looking at a small loop around $t=0$ has the four points corresponding to $x^4 + t^4$ making a full counter-clockwise circle. As we take a circle in the $t$-plane centered on the origin and increase its radius, we can watch when there is simple branching in $\Delta_s$. This occurs when the radius is again at $4(s/5)^{5/4}$ for $t$ at the four values of $t = -4 (s/5)^{(5/4)}$. Increasing $s$ splits the single branching point of $\Delta$ over $t=0$ into two points for $\Delta_s$, one at $x=-s$ of multiplicity one and the other at $x=0$ of multiplicity 4. Moving from $t=0$ to $t = 4(s/5)^{5/4}$ along the real axis, the point at $x = 0$ splits into four distinct points roughly lying in the centers of the four quadrants of $\CC$. As the angle of $t$ increases from $0$ to $2\pi$, those four points rotate past the negative real axis and collide with the point that initially corresponded to $x=-s$. These are the four points of simple branching and lying above them in $V(f_s)$ and $X_s$ are four Lefschetz singular fibers.

To keep track of where the branching occurs with the goal of producing a positive word description of the resulting Lefschetz fibration, we choose paths in the $t$-plane connecting the reference fiber $t=1$ to a place where $\Delta_s$ branches and then pull the arc of collision back. Doing this for all five singularities, we obtain the factorization shown in Figure~\ref{fig:phi4T}. Choosing a different identification of the skeleton as was used for $\phi_1$ gives us a nice and particularly useful factorization. The separating twist corresponds to the points involved in the $x^4 + t^4$ singularity over the origin and lifts to the the twists $\tau_5$ and $\tau_{5'}$ (as labeled in Figure~\ref{fig:gens}). We will refer to it by the shorthand $\phi_{\tilde{4}}$: $$\phi_{\tilde{4}} = \tau_5 \tau_{5'} \tau_4 \tau_3 \tau_2 \tau_1 .$$

\begin{figure}
    \centering
    \includegraphics[width=4.5in]{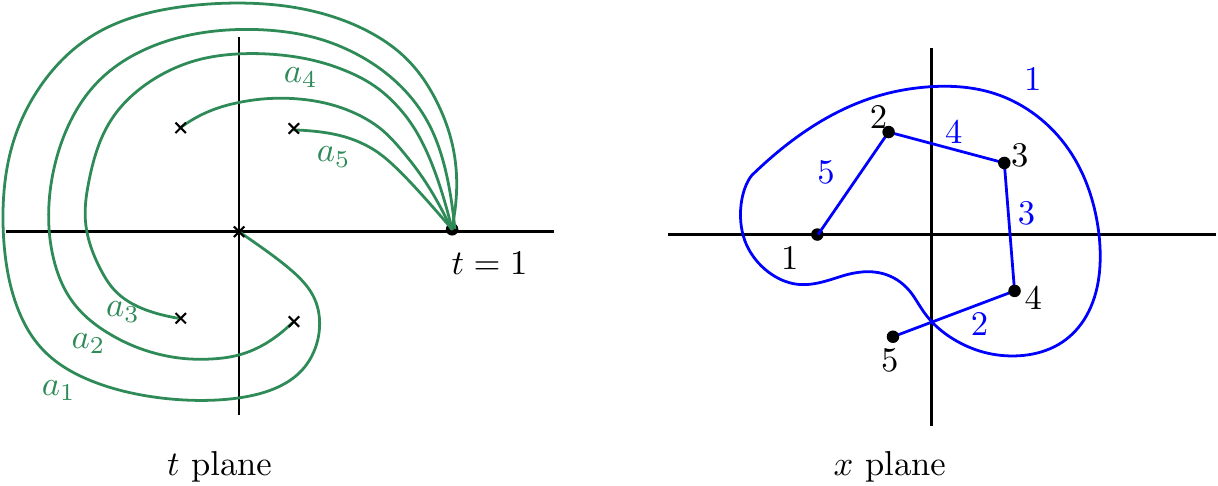}
    \caption{Paths in the $t$-plane used to find the factorization of $\phi_1^4$, along with the corresponding factorization for $\phi_{\tilde{4}}$. The singularity over $t=0$ (and corresponding to the path labeled $a_1$), is a full twist about the enclosed points and lifts to a pair of Dehn twists that separate off a genus one surface with two boundary components.}
    \label{fig:phi4T}
\end{figure}

This word has length 6. The corresponding Lefschetz fibration yields a manifold with $b_2 =2$ and which is negative definite, and so $K^2$ agrees with that of the resolution. (See \ref{negdef}.) By Laufer, this deformation yields a simultaneous resolution and so this Lefschetz fibration is a deformation of the fibration on $y^2 = x^5 +t^4$.

\smskip

\noindent \textbf{Case} {\boldmath $y^2= x^5 +t^5$}:
This is another important case but one which actually reduces to previous work. The monodromy of the open book at the boundary \emph{is} the hyperelliptic involution on the genus 2 surface with one boundary component. The branch curve $x^5 + t^5$ is the cone on the $(5,5)$ torus knot, the 5-braid with a full twist. We'll see that the resolution is the Lefschetz fibration associated to the word $$I =\tau_1 \tau_2 \tau_3 \tau_4 \tau_5 \tau_{5'} \tau_4 \tau_3 \tau_2 \tau_1 $$

This splitting is rather straightforward. Deform $y^2= x^5 +t^5$ to $y^2= x^5 +t^4(t+s)$. This has two singularities, one of type $\phi_{\tilde{4}}$ over $t=0$ and one of type $\phi_1$ over $t = -s$. We choose paths in the plane from $t=1$ to the singularities so that the path to $t=-s$ is to the right and the path to $t=0$ lies one the real line. In order to invoke this particular factorization, we choose to use the same identifications of the fiber over $t=1$ with the frame for $\phi_1^4 = \phi_{\tilde{4}}$, so that the monodromy splits directly into a product of first a conjugate of $\phi_1$ and second $\phi_{\tilde{4}}$. The particular conjugate of $\phi_1$ is just a counterclockwise rotation by $4\pi/5$ and so by Lemma~\ref{lm:phi1rotate} is Hurwitz equivalent to the standard factorization but written on the new skeleton. This yields the factorization:
$$\phi_1^5 = \phi_1 \phi_{\tilde{4}} = I = \tau_1 \tau_2 \tau_3 \tau_4 \tau_5 \tau_{5'} \tau_4 \tau_3 \tau_2 \tau_1.$$

The word $\tau_1 \tau_2 \tau_3 \tau_4 \tau_5 \tau_{5'} \tau_4 \tau_3 \tau_2 \tau_1 $ is negative definite with $b_2 = 6$. This agrees with the resolution given in Section~\ref{phi_1s} and so this admits a simultaneous resolution and the Lefschetz fibration given by the positive word $I$ is a deformation of the fibration on the resolution.

\smskip

\noindent \textbf{Case} {\boldmath $y^2= x^5 +t^6$}:
Just as in the previous case, this splits via $y^2= x^5 +t^5 (t+s)$ into a type $I$ singularity at $t=0$ and type $\phi_1$ singularity at $t=-s$. As before, something strange can occur with the identification of the singularity at $t=-s$ which can yield a term of the form $\phi_1^k(\phi_1)$, but this is Hurwitz equivalent to $\phi_1$. This yields the factorization $\phi_1^6 = \phi_1 I$. The Lefschetz fibration corresponding to the positive word $\phi_1 I$ is negative definite with $b_2=10$. This agrees with the resolution given in Section~\ref{phi_1s} and so this admits a simultaneous resolution and the Lefschetz fibration given by the positive word $$\phi_1^6 = \phi_1 I$$ is a deformation of the fibration on the resolution.

\smskip

\begin{lemma}\label{lm:phi4Trotate} The factorization $\phi_{\tilde{4}} = \tau_5 \tau_{5'} \tau_4 \tau_3 \tau_2 \tau_1$ is Hurwitz equivalent to all other factorizations obtained by rotating the corresponding braid factorization by some multiple of $2\pi/5$ and lifting via the branched double cover. More explicitly, $\phi_1^k(\phi_{\tilde{4}})$ is Hurwitz equivalent to $\phi_{\tilde{4}}$.
\end{lemma}

\begin{proof} As before, this really follows from the equality of the monodromies $\phi_{\tilde{4}} = \phi_1^4$ and the fact that $I = \phi_1^5$ is the hyperelliptic involution and so commutes with everything, but we give a direct proof using Hurwitz moves. A diagrammatic description is shown in Figure~\ref{fig:phi4Trotate}. First we conjugate all the braid half-twists by the Dehn twist about the closed curve. Then we apply Hurwitz moves, sliding the first arc past the remaining arcs until it again lies on the pentagon. Finally, since the product of all the half-twists is the counterclockwise rotation by $2 \pi/5$, we conjugate the Dehn twist by this product of half twists to get the final and desired factorization. 
\end{proof}

\begin{figure}
    \centering
    \includegraphics[width = 6in]{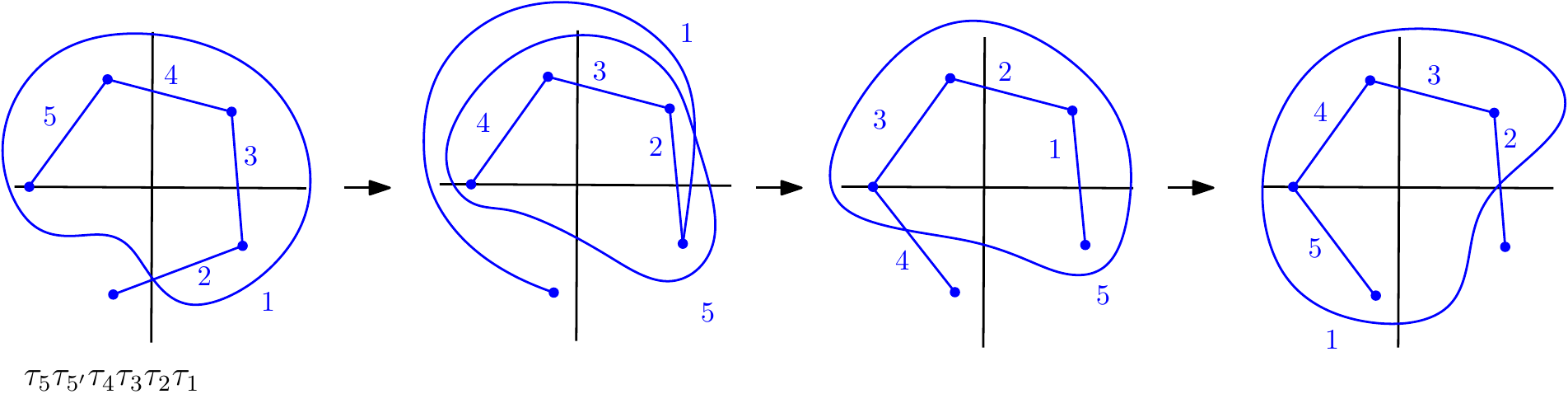}
    \caption{The sequence of Hurwitz moves that relates $\phi_{\tilde{4}}$ and $\phi_1(\phi_{\tilde{4}})$
    (as drawn in the quotient plane). First we conjugate things by the Dehn twist indicated by the closed curve. Then we move the first arc past the rest until it again lies on the pentagon. Lastly we conjugate the Dehn twist past all the half twists. The labeling of the curves indicates their order in the factorization.}
    \label{fig:phi4Trotate}
\end{figure}

\noindent \textbf{Case} {\boldmath$y^2= x^5 +t^8$}:
This splitting is again similar to previous cases. Via the deformation $y^2= x^5 +t^4 (t-is)^4$, this splits into two type $\phi_{\tilde{4}}$ singularities, one at $t=0$ and one at $t=is$. The identification of the fiber 
at $t = i s$ with the original $x^5 + t^4$ singularity is by some rigid rotation in the plane and so this gives us the word $\phi_1^k(\phi_{\tilde{4}}) \phi_{\tilde{4}}$ but by the previous lemma, this is Hurwitz equivalent to the factorization $\phi_{\tilde{4}}^2$. The Lefschetz fibration corresponding to the positive word $\phi_{\tilde{4}}^2$ is negative definite with $b_2=8$. This agrees with the resolution given in Section~\ref{phi_1s} and so this admits a simultaneous resolution and the Lefschetz fibration given by the positive word $\phi_{\tilde{4}}^2=(\tau_5 \tau_{5'} \tau_4 \tau_3 \tau_2 \tau_1 )^2$ is a deformation of the fibration on the resolution.

\smskip

\noindent \textbf{Case} {\boldmath$y^2= x^5 +t^9$}:
This splits via the deformation $y^2= x^5 +t^5 (t-is)^4$. There is one singularity of type $I$ at $t=0$ and one of type $\phi_{\tilde{4}}$ at $t=is$. The singularity at $t=is$ 
has the form $y^2 = x^5 + (-is)^5 (t-is)^4$, which is again rotation of the factorization  $\phi_{\tilde{4}}$. This splitting yields a factorization of the form $\phi_1^k(\phi_{\tilde{4}}) I$, which by Lemma~\ref{lm:phi4Trotate} is Hurwitz equivalent to $\phi_{\tilde{4}} I$. The Lefschetz fibration corresponding to the positive word $\phi_{\tilde{4}} I $ is negative definite with $b_2=12$. This agrees with the resolution given in Section~\ref{phi_1s} and so this admits a simultaneous resolution and the Lefschetz fibration given by the positive word $\phi_{\tilde{4}} I $ is a deformation of the fibration on the resolution.

\noindent \textbf{Case} {\boldmath$y^2= x^5 +t^{10}$}:
In the resolution of $\phi_1^{10}$, if we dropped the arrow, at the end of the successive blow-downs we obtain a genus 2 surface of self intersection $-1$ (see the last row in Figure \ref{phi1ktable}). The resolution of this singularity then has $b_1 = 4$ and $b_2 = 1$, and so must be given by a single separating Dehn twist. The monodromy is a boundary twist, $\phi_1^{10} = \tau_{\bdry}$, and so the factorization must be a boundary twist as well.

\subsection{\texorpdfstring{\boldmath $\phi_3^k$}{phi3k}}
This is the family $y^2=x^6+t^k$.

\smskip

\noindent \textbf{Case} {\boldmath $y^2=x^6+t$}: This case is analogous, indeed nearly identical, to the $\phi_1$ case. We split this via the deformation $y^2 = x^6+sx+t$.  Duplicating the picture for $\phi_1$, one degree larger, gives the braid description for $\phi_3$ of $a_1a_2a_3a_4a_5$. Then the positive word for the Lefschetz fibration is $\phi_3 = \tau_1 \tau_2 \tau_3 \tau_4 \tau_5$ with the skeleton shown in Figure~\ref{fig:phi3}. We note that, just as in the many of the other cases, the total braid monodromy of this surface is a counterclockwise rotation of $2\pi/6$ and so by changing the reference arc to the singularity, we can transform the skeleton by any power of this rotation and this gives a Hurwitz equivalent positive word.
\smskip

\begin{lemma} \label{lem:phi3rotate} The factorization $\phi_3 = \tau_1 \tau_2 \tau_3 \tau_4 \tau_5$ is Hurwitz equivalent to all other factorizations of $\phi_3$ given by rotating the braid factorization in Figure~\ref{fig:phi3} by some multiple of $2\pi/6$ and lifting via the branched double cover. More explicitly, since the monodromy $\phi_3$ is represented by this rigid rotation by $2\pi/6$, we have that $\phi_3^k (\phi_3)$ is Hurwitz equivalent to $\phi_3$.
\end{lemma}

\begin{proof} The proof is the immediate analogue of the proof of Lemma~\ref{lm:phi1rotate}.
\end{proof}

%\begin{figure}\centering \includegraphics[width=6in]{figs/Phi3.png}\caption{The positive factorization for $\phi_3$.} \label{fig:phi3}\end{figure}

\begin{figure}
    \centering
    \includegraphics[width=4.3in]{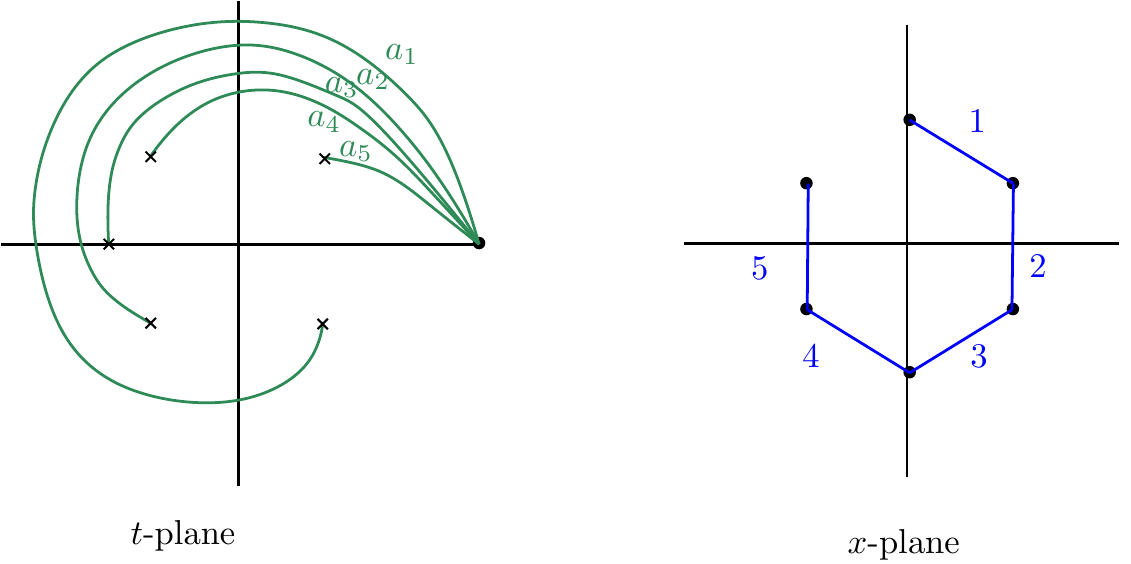}
    \caption{The positive factorization for $\phi_3$. On the left we see the paths in the $t$-plane from the reference fiber to the singular fibers. On the right, the arcs along which the corresponding braid half twists occur. The Dehn twists in the Lefschetz fibration are lifts of these braid half twists to the double cover of the disk branched over the marked points.}
    \label{fig:phi3}
\end{figure}

\noindent \textbf{Case} {\boldmath $y^2=x^6+t^2$}: 
The resolution of this singularity has $b_1 = 0$ and $b_2 = 5$. The length of the corresponding positive word is $10$, which we can achieve by splitting the singularity as $y^2=x^6+t(t+s)$, yielding two type $\phi_3$ singularities, one at $t=0$ and one at $t=-s$. The identification of the singularity at $t=0$ along the from $t=1$ to $t=0$ along the real line agrees with the work done on the previous case. The singularity at $t=-s$ along a path with positive imaginary coordinate gives us some conjugate $\phi_3^k(\phi_3)$. By Lemma~\ref{lem:phi3rotate} this is Hurwitz equivalent to our preferred factorization $$\phi_3^2 =(\tau_1 \tau_2 \tau_3 \tau_4 \tau_5)^2.$$ 
\smskip

\noindent \textbf{Case} {\boldmath $y^2=x^6+t^4$}: 
The resolution of this singularity has $b_1 = 0$ and $b_2 = 5$ so the Lefschetz fibration will correspond to a word of length 10. We split this via the deformation $(x^2-s)x^4+t^4$ and obtain one $x^4+t^4$  singularity at $(0,0)$, which in the branched cover consists of two Dehn twists that form a separating pair, and 8 other branch points. Before the deformation above, the branch locus $\Delta$ intersects the fiber above $t=1$ in the 6th roots of $-1$. The deformation takes this hexagon of points and decreases the norm of their real part while increasing the norm of their imaginary part, mildly stretching the hexagon vertically. The singularity above $t=0$ is a full twist among the four points closest to the origin, the points which do not lie on the imaginary axis. The remaining branch point singularities come in four matched pairs, each pair lying in a single $t$-fiber and having angular components $\pi/4, 3\pi/4, 5\pi/4,$ and $7\pi/4$. Around the singularity over the ray at $\pi/4$, the braiding consists of two of the four center points, the points in the first and third quadrants in the cartesian plane, braiding with the two points along the imaginary axis. As you continue in the $t$-plane in a circle around the origin, this movie repeats but after having rotated the hexagon. To get the factorization, then we choose paths outside of the disk containing the singularities and pull back the braiding along these paths. The picture is rotationally symmetric and we see the factorization comes as a list of four pairs of braid half twists (where the first and last pair end up being the same).

Tracing out the matched branching, we get the factorization of $\phi_3^4$ into 10 twists, where all curves are shown in Figure~\ref{fig:phiAmonodromy}. We then choose the same skeleton chosen for $\phi_3$ to align our image with the standard generators to get the factorization:

$$\phi_3^4 = \phi_A = \tau_1 \tau_4 \tau_3 \tau_{a_1} \tau_2 \tau_5 \tau_1 \tau_4 \tau_{b_1} \tau_{b_1'}$$

where the curves $a_1$, $b_1$, and $b_1'$ on the genus 2 surface are shown in Figure~\ref{fig:phiAmonodromy}.

\smskip

%\begin{figure}\centering\includegraphics[width=3in]{figs/Phi341.png} \\\includegraphics[width=6in]{figs/Phi342.png} \\\includegraphics[width=4in]{figs/Phi343.png} \\\caption{The positive factorization for $\phi_3^4$.}\label{fig:phi34}\end{figure}
\begin{figure}
    \centering
    \includegraphics[width=2in]{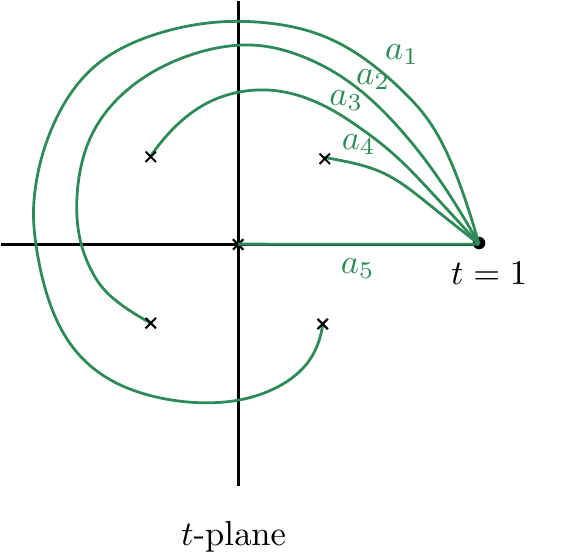} \\
    \includegraphics[width=6in]{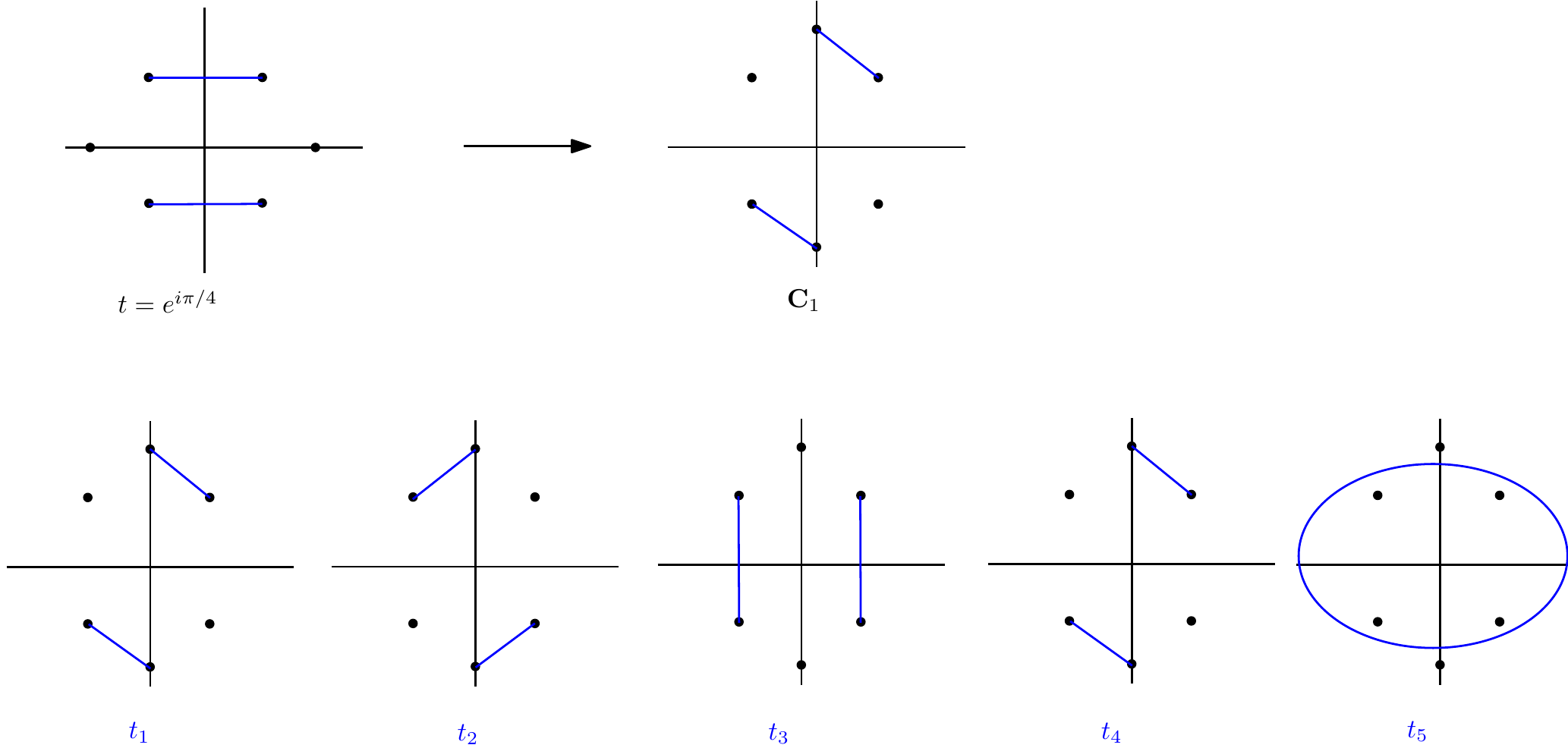} \\
    \vspace{0.2in}
    \includegraphics[width = 4.5in]{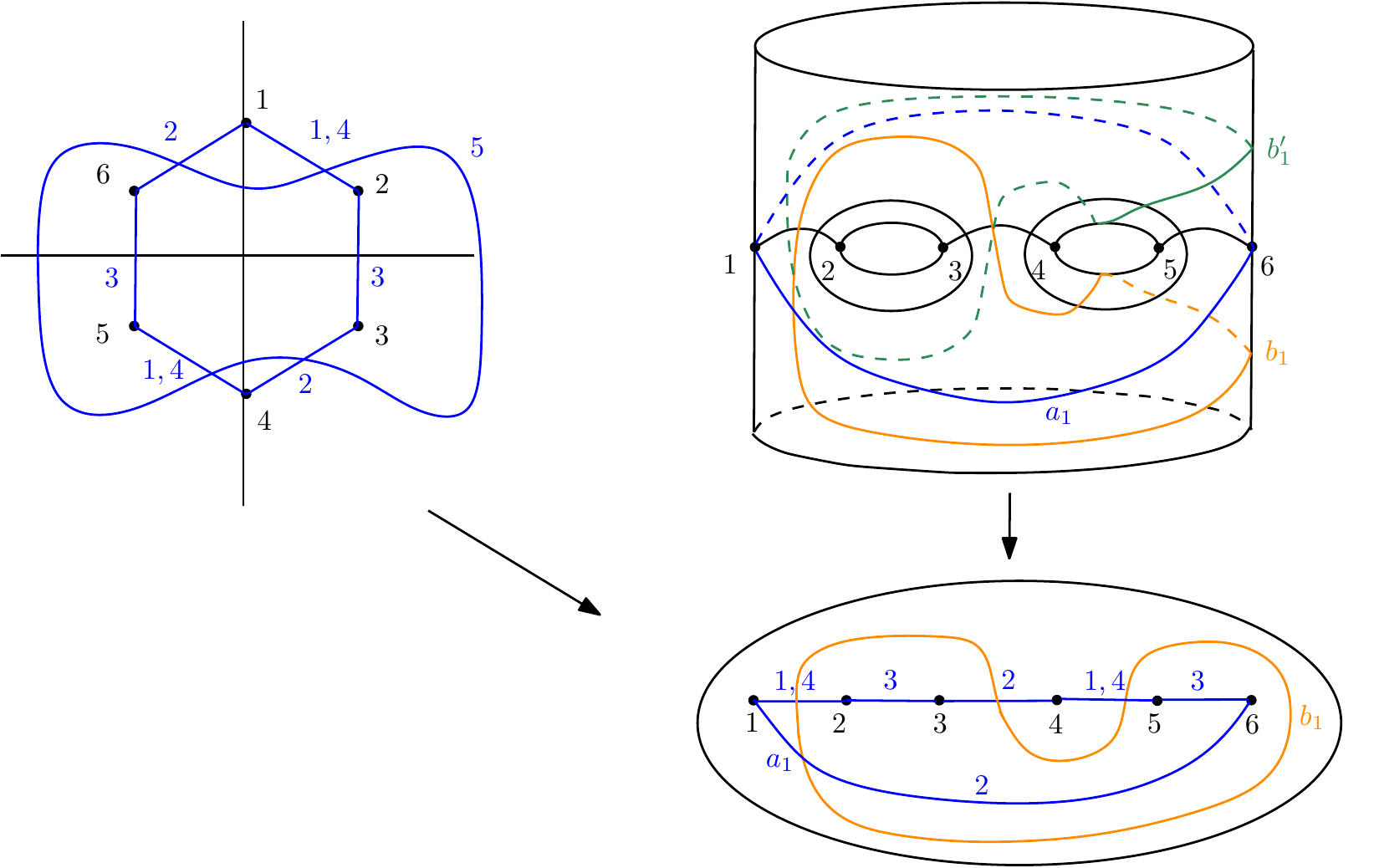}

    \caption{The positive factorization $\phi_A$ for $\phi_3^4$. The factorization consist of 4 pairs of twists, each pair along opposite sides of the hexagon, along with a pair of twists that separate off a genus 1 subsurface. We draw the picture of the branch locus for the hyperelliptic quotient and then lift the factorization to the genus 2 surface.}
    \label{fig:phiAmonodromy}
\end{figure}

\noindent \textbf{Case} {\boldmath $y^2=x^6+t^5$}:

The resolution of this singularity has $b_1 = 0$ and $b_2 = 10$ so the Lefschetz fibration will correspond to a word of length 15. We split this via the deformation $x^6+t^4(t+s)$ and obtain two singularities, one of type $\phi_3=x^6 + t$ at $t=-s$ and one of type $\phi_3^4 = x^t + t^4$ at $t=0$. We choose paths so that the path from $t=1$ to $t=0$ lies on the real axis, and the path to $t = -s$ lies to the right. The singularity at $t = -s$ agrees with $\phi_3$ after rotation by some power of $\phi_3$, and pulling back also changes this identification by some power of $\phi_3$, so this gives us the positive word $\phi_3^k(\phi_3) \phi_A$. By Lemma~\ref{lem:phi3rotate}, this is then Hurwitz equivalent to the word
$$\phi_3^5 = \phi_3 \phi_A.$$

\smskip

\noindent \textbf{Case} {\boldmath$y^2= x^6 +t^{6}$}:
In the resolution of $\phi_3^{6}$, if we dropped the two arrows, at the end of the successive blow-downs we obtain a genus 2 surface of self intersection $-2$ (see the last row in Figure \ref{phi3ktable}). The resolution of this singularity then has $b_1 = 4$ and $b_2 = 1$, and so must be given by two Dehn twists along two homologous curves. The monodromy is then the boundary multitwist, $\phi_3^{6} = \tau_{\bdry}$, and the factorization must be the boundary multitwist as well.

\smskip

\subsection{\texorpdfstring{\boldmath $\phi_2^k$}{phi2k}}
% 1 boundary component in this family.

The next family we look at is $y^2 = x(x^4 + t^k)$. As before, there is a $k$-fold branched cover over the principal case of $k=1$, $y^2 = x(x^4 + t)$, branched over the singular fiber at $t=0$. We'll denote the monodromy of the principal case as $\phi_2$ and we'll show that the resolution of that singularity corresponds to the Lefschetz fibration with word $\phi_2 = (\tau_1 \tau_1 \tau_2 \tau_3 \tau_4).$ The monodromy for the open book on the link of the singularity is then just the $k$th power of $\phi_2$, $\phi_2^k$. In each case, we'll find the positive word that describes the Lefschetz fibration on the resolution.
For this family, the fiber over $t=1$ is the curve $y^2 = x(x^4+1)$, which is hyperelliptic and the branched double cover of $\CC$ branched over the points $x(x^4+1)=0$. We choose to identify this curve with the standard curve shown in Figure~\ref{fig:gens} using the skeleton shown in Figure~\ref{fig:phi2skeleton}.

\smskip
\begin{figure}
    \centering
    \includegraphics[width =2.3in]{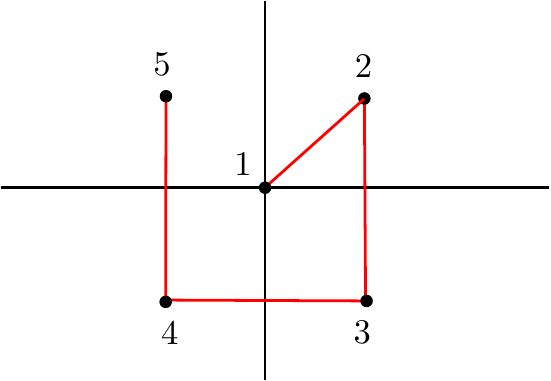}
    \caption{The identification of $y^2 = x(x^4+1)$ with the standard genus 2 surface, $\Sigma_{2,1}$, from Figure~\ref{fig:gens}.}
    \label{fig:phi2skeleton}
\end{figure}

\noindent \textbf{Case} {\boldmath $y^2= x(x^4 +t)$}:
The curve $\Delta = \{ (t,x) \big| x(x^4 +t) =0\}$ has two components, the line $x = 0$ and the quartic $x^4 + t$. These are smooth curves that intersect transversely at $(0,0)$ but the quartic component is tangent to the fiber $t=0$. 
Choosing the deformation $(x-is)(x^4+t)$ deforms $\Delta$ via $\Delta_s$ so that the intersection between the two components 
occurs transverse to the fibration and with local model $x^2 + t^2$. The double cover of $\CC^2$ branched over this singular surface is the standard double point for the duVal singularity $A_1$. Topologically this is resolved by gluing in a $-2$ sphere and the corresponding Lefschetz fibration has the two parallel Dehn twists. For $s \neq 0 $, the fibration on $X_s$ has two singular fibers, one with local model $x^4 + t$ and the other as discussed above. We saw in Section~\ref{sec:lower genus} that $x^4 + t$ splits into the Lefschetz fibration with monodromy $\tau_1 \tau_2 \tau_3$ so we just need to determine how the Dehn twists from the $A_1$ singularity and the Dehn twists from the $x^4 + t$ singularity interact.

To do this we explicitly track the monodromies of the two singularities as measured from $t=1$. The singular fibers lie at $t=0$ and $t=-s^4$. Taking the paths as shown in Figure~\ref{fig:phi2mon}, the singularity at $t = -s^4$ yields the double point with monodromy $\phi_1^3$. Traversing the singularity at $t=0$ give a counterclockwise quarter rotation among the 4th roots of $-1$. The arcs in the factorization of the singularity $x^4 + t = 0$ are identified with the generators $\tau_2$, $\tau_3$ and $\tau_4$. Together this gives a Lefschetz fibration on $X_s$ with monodromy $$\phi_2 = \tau_1\tau_1\tau_2\tau_3\tau_4.$$ Comparing with the resolution, this factorization yields a manifold with $b_1 = 0$ and $b_2 = 1$ which agrees with the resolution from  Section~\ref{sec:phi2kres}, exactly as required.
\begin{figure}
\centering \includegraphics[width=5.7in]{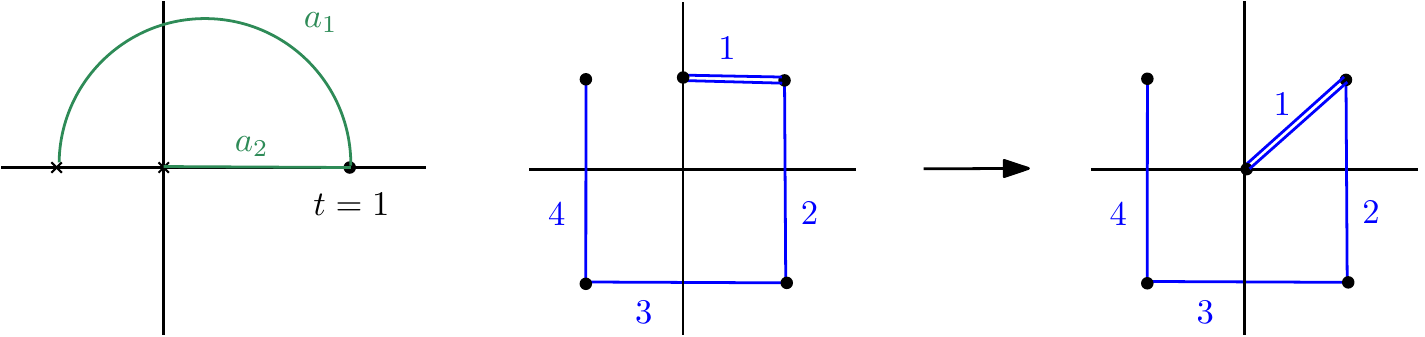}
\caption{On the left we show the two paths to the degenerate singularities in the deformation $(x-s)(x^4+t)$ in the evaluation of $\phi_2$. On the right, the two pictures that correspond to factorization of $\phi_2$ measured at $t=1$.  The left is the factorization for the deformed, with $s$ positive real, and on the right is that factorization pulled back to the fiber $t=1$ when $s=0$.} \label{fig:phi2mon}
\end{figure}

\smskip
\begin{lemma} \label{lem:phi2rotate} Just as in the case of $\phi_1$, this factorization of $\phi_2$ is Hurwitz equivalent to all of the other factorizations obtained by rotating the plane through a multiple of $2\pi/4$. Equivalently, the factorizations $\phi_2^k(\phi_2)$ and $\phi_2$ are Hurwitz equivalent. The statement also holds for the factorization $\phi_{\tilde{2}}$: $\phi_2^k(\phi_{\tilde{2}})$ and $\phi_{\tilde{2}}$ are Hurwitz equivalent.
\end{lemma}

\begin{proof} 
For the factorization of $\phi_2 = \tau_1\tau_1\tau_2\tau_3\tau_4$, we begin by conjugating the double twist in the first position past the next twist, taking us to the second image. Next sliding the curve now labeled 1 past all the other curves takes us to the third and final configuration, $\phi_2^{-1}(\phi_2)$. 

For the factorization of $\phi_{\tilde{2}} = \tau_5\tau_{5'}\tau_4\tau_3\tau_2$,  we start by sliding the full twist 1 over the twist 2. Then reversing that, sliding the halftwist now labeled 1 over the full twist now labeled 2. Then we slide that halftwist labeled 2 across 3 and 4 to arrive at the factorization $\phi_2(\phi_{\tilde{2}})$. (See Figure \ref{fig:phi2rotate}).
\end{proof}

\begin{figure}
    \centering
   \includegraphics[width = 5 in]{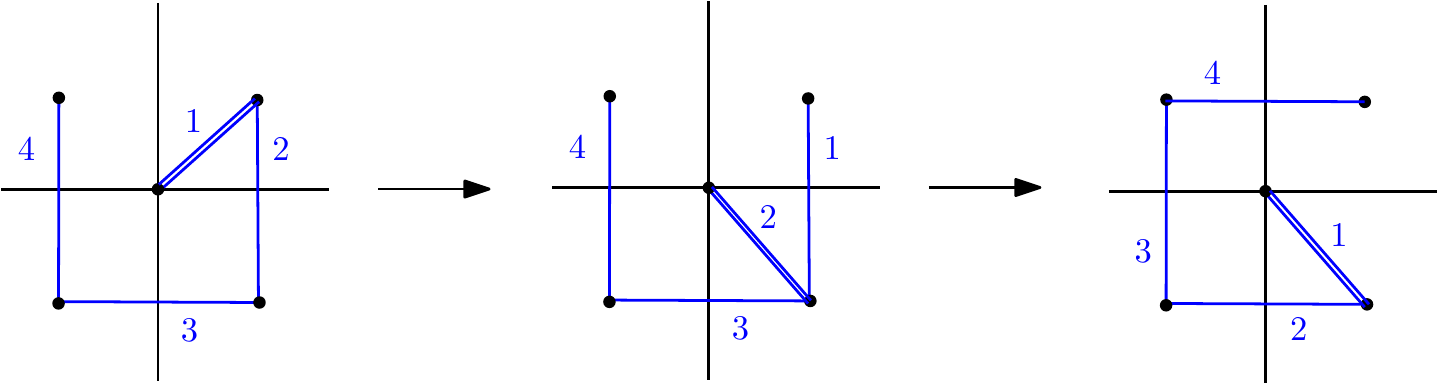}\\
    \vspace{0.2in}
    \includegraphics[width = 6 in]{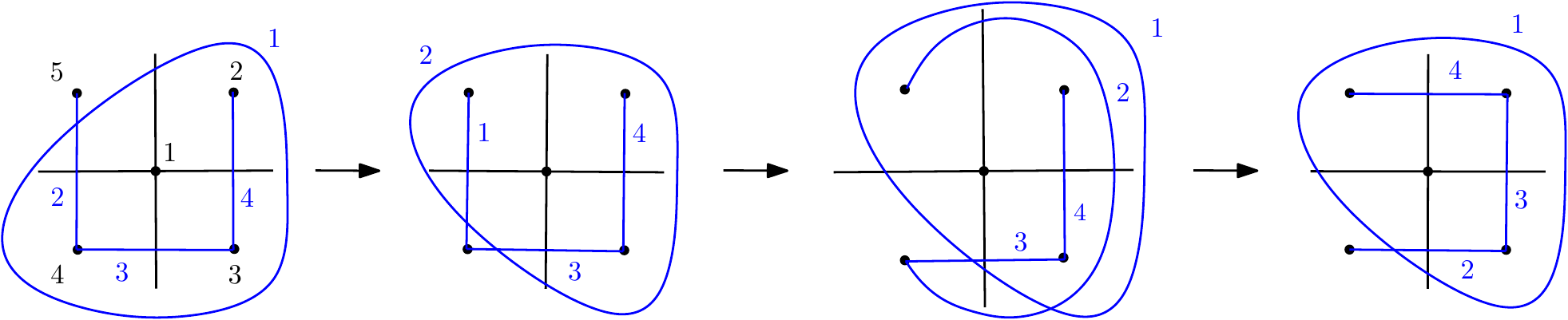}
    \caption{The sequence of Hurwitz moves that relates the factorization $\phi_2$ and $\phi_2^{-1}(\phi_2)$. Also shown is the sequence of moves that relates $\phi_{\tilde{2}}$ and $\phi_2^{-1}( \phi_{\tilde{2}})$. These moves are discussed in Lemma~\ref{lem:phi2rotate}.}
    \label{fig:phi2rotate}
\end{figure}

\noindent \textbf{Case} {\boldmath $y^2= x(x^4 +t^2)$}: 
Deforming by $x(x^4 +t(t-s))$ yields fibrations $X_s$ with two type $\phi_2$ singularities, at $t=0$ and $t=s$. Resolving these yields a surface with $b_1=0$ and $b_2 = 6$, matching the resolution of $y^2= x(x^4 +t^2)$ (see Section \ref{sec:phi2kres}). The fibration on $\tilde{X}$ therefore deforms into the Lefschetz fibration given by the positive word $\phi_2^2 = (\tau_1\tau_1\tau_2\tau_3\tau_4)^2$.

\smskip

\noindent \textbf{Case} {\boldmath $y^2= x(x^4 +t^3)$}: This resulting monodromy looks very similar to that of $\phi_1$, and indeed both are isomorphic after capping. We call the resulting factorization $\phi_{\tilde{2}}$.

Deform the branch curve $\Delta$ to $x((x-s)x^3+t^3)$. This yields fibrations on $X_s$ with four singular fibers. One singularity is of type $x(x^3 + t^3)$ located at $t=0$. As was shown in Section~\ref{sec:lower genus}, this corresponds to a Dehn twist that separates off a genus 1 subsurface. The other singularities are simple braidings. The singular fiber at $t=0$ resolves to a pair of Dehn twists along curves that separate off a genus 1 surface with two boundary components. The other three singularities occur at the third roots of $3^3/4^4 s^4$.

Tracing out the monodromy along the paths in the $t$-plane as shown in Figure~\ref{fig:phi2Tmon}, we see how these four Dehn twists are arranged relative to each other. This yields a fibration with monodromy $$\phi_{\tilde{2}} = \tau_5\tau_{5'}\tau_4\tau_3\tau_2.$$ The total space has $b_1 = 0$ and $b_2 = 1$, matching the resolution in Figure~\ref{phi2ktable}, and so this is a deformation of the resolution.

%\begin{figure}\centering\includegraphics[width=5in]{figs/Phi2p}\caption{The monodromy of $\phi_{\tilde{2}}$.} \label{fig:phi2Tmon}\end{figure}

\begin{figure}
    \centering
    \includegraphics[width=4in]{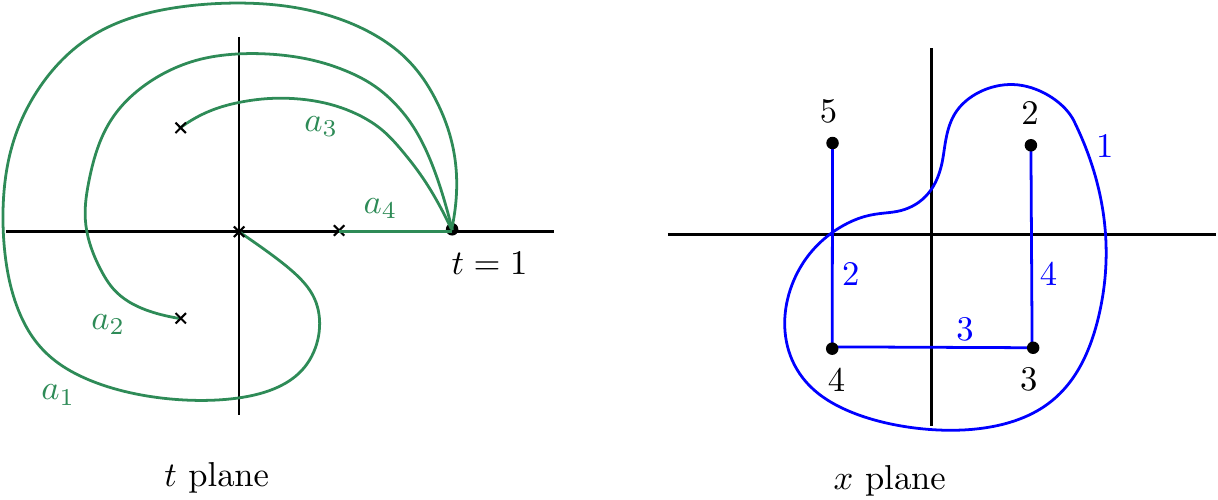}
    \caption{The monodromy of $\phi_{\tilde{2}}$ and the paths in the $t$-plane that realize it.}
    \label{fig:phi2Tmon}
\end{figure}
\smskip

\noindent \textbf{Case} {\boldmath $y^2= x(x^4 +t^4)$}:
This Lefschetz fibration turns out to be the same as $\phi_1^5 = I =\tau_1\tau_2\tau_3\tau_4\tau_5\tau_{5'}\tau_4\tau_3\tau_2\tau_1$ and the argument is very similar. We split $x(x^4 +t^4)$ into $x(x^4 +(t+s)t^3)$, with a type $\tilde{\phi_2}$ singularity ($x(x^4 +t^3)$) at $t=0$ and a type $\phi_2$ singularity ($x(x^4 +t)$) at $t = - s$ (up to some identification). Using the path from $t=1$ to $t=0$ along the real line, we can identify everything from the $\phi_{\tilde{2}}$ singularity from the $y^2 = x(x^4 + t^3)$ case without change. If we take the path with positive imaginary value to the $\phi_2$ singularity at $t= -s$, we pull back the factorization of $\phi_2$ up to conjugation by some power of $\phi_2$. By Lemma~\ref{lem:phi2rotate}, this conjugate is Hurwitz equivalent to the factorization $\phi_2$.
This gives us the factorization: $\phi_2^4 = \phi_2 \phi_{\tilde{2}} = \tau_1\tau_1\tau_2\tau_3\tau_4\tau_5\tau_{5'}\tau_4 \tau_3 \tau_2$.
To get to the desired word of $\tau_1\tau_2\tau_3\tau_4\tau_5\tau_{5'}\tau_4\tau_3\tau_2\tau_1$, we can apply Hurwitz moves to the first instance of $\tau_1$ to slide it past the rest of the word. The rest of the word is $I \tau_1^ {-1}$ and since $\tau_1$ is hyperelliptic, after conjugating we get back a Dehn twist along the same curve. Thus we have 

$$\phi_2^4 = \phi_2 \phi_{\tilde{2}} = \tau_1\tau_1\tau_2\tau_3\tau_4\tau_5\tau_{5'}\tau_4 \tau_3 \tau_2 \equiv \tau_1\tau_2\tau_3\tau_4\tau_5\tau_{5'}\tau_4 \tau_3 \tau_2\tau_1 = I.$$

This is exactly the word found for $f(x,t) = x^5 +t^5$. 

\smskip

\noindent \textbf{Case} {\boldmath $y^2= x(x^4 +t^5)$}:
The resolution of this singularity has $b_1 = 0$ and $b_2 = 11$ and so its Lefschetz fibration will have a positive word factorization of length 15. Using the deformation $y^2= x(x^4 +t^4(t+s))$, this splits into singularities of types $y^2= x(x^4 +t^4)$ at $t=0$ and $y^2= x(x^4 +(t+s))$ at $t=-s$. Choosing paths in the plane so that the path to $t=s$ is to the right and the path to $t=0$ lies on the real line, then this writes our positive factorization as $\phi_2^5 = \phi_2^k(\phi_2) I$ which is Hurwitz equivalent to $\phi_2 I$ by Lemma~\ref{lem:phi2rotate}.

\smskip

\noindent \textbf{Case} {\boldmath $y^2= x(x^4 +t^6)$}:
The resolution has $b_1 = 0$ and $b_2 = 6$, so our Lefschetz fibration will have word length 10. Splitting by $x(x^4+t^3(t-s)^3)$ yields two singularities of type $x(x^4+t^3)$, one at $t=0$ and one at $t=s$, giving a factorization $\phi_{\tilde{2}}^2$.

\smskip

\noindent \textbf{Case} {\boldmath $y^2= x(x^4 +t^7)$}:
The resolution has $b_1 = 0$ and $b_2 = 11$, so our Lefschetz fibration will have word length 15. We split via the deformation $x(x^4+t^4(t+s)^3)$ to get a singularity of type $I$ at $t=0$ and $\phi_{\tilde{2}}$ at $t=s$. We continue as in the case $y^2 = x(x^4 + t^5)$, invoking Lemma~\ref{lem:phi2rotate} to get the positive factorization $\tilde{\phi_2} I$.

\smskip

\noindent \textbf{Case} {\boldmath $y^2= x(x^4 +t^8)$}:
In the resolution of $\phi_2^{8}$ we obtain a genus 2 surface of self intersection $-1$ (see the last row in Figure \ref{phi2ktable}). The resolution of this singularity then has $b_1 = 4$ and $b_2 = 1$, and so must be given by a single separating Dehn twist. The monodromy is a boundary twist, $\phi_2^{8} = \tau_{\bdry}$, and so the factorization must be a boundary twist as well.

\smskip

\subsection{\texorpdfstring{\boldmath $\phi_4^k$}{phi4k}}
Lastly, we take the family $y^2= x(x^5 +t^k)$. In the mapping class group, the monodromy for $k=1$ can be written as $\tau_1 \tau_1 \tau_2 \tau_3 \tau_4 \tau_5$ which is a 5th root of the boundary multitwist. 

% 2 boundary components in this family.
\smskip

\noindent \textbf{Case} {\boldmath $y^2= x(x^5 +t)$}: 
Similar to the $\phi_2$ case, the branch curve defined by $x(x^5 +t)$ has two components that meet at $(0,0)$. We split this singularity via the deformation $y^2 = (x+s)(x^5+t)$ which moves the point of intersection to $(-s, s^5)$ where the two components intersect transversely. After the deformation we see branching at $t = 0$ and $t = s^5$, splitting the singularity into a $x^5+t$ singularity at $t=0$ and a transverse double point at $t = s^5$. We understand each of these monodromies separately, and, as in the $\phi_2$ case, we need to trace through how these two identifications match up. For ease, we take $s$ to be a positive real perturbation and, use the arcs shown in Figure~\ref{fig:phi4}. The double point gives us a factor of $\tau_1^2$ and the $x^5 + t$ singularity at the origin splits into something which we can identify with $\phi_1$, $\tau_2 \tau_3 \tau_4 \tau_5$ using the skeleton chosen. This gives us the monodromy factorization of  $$\phi_4 = \tau_1 \tau_1 \tau_2 \tau_3 \tau_4 \tau_5.$$ This yields a total space with $b_1 = 0$ and $b_2 = 1$, which agrees with that of the resolution and so gives us a flat deformation of the resolution.

%\begin{figure}  \centering \includegraphics[width=5in]{figs/Phi4.png}  \caption{The monodromy of $\phi_4$.} \label{fig:phi4}\end{figure}
\begin{figure}
    \centering
 \includegraphics[width=3in]{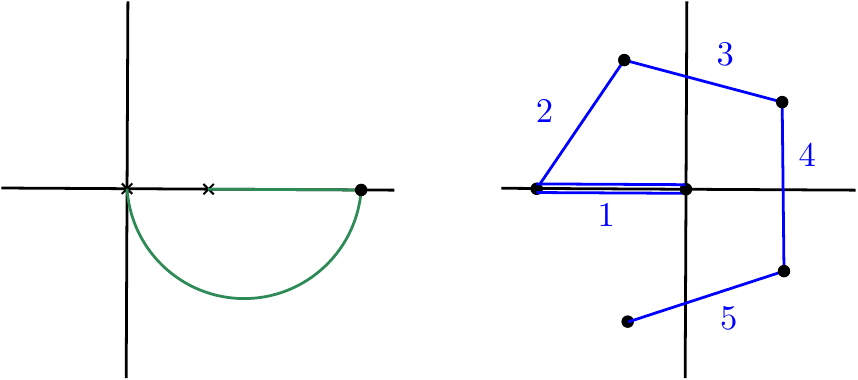} \\ 
 
 \includegraphics[width=1.5in]{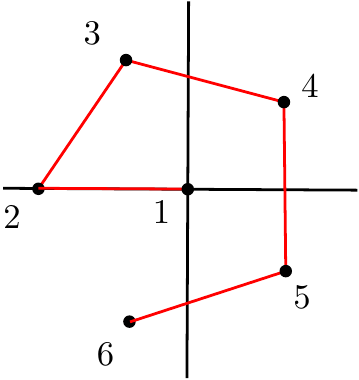}

    \caption{The top row shows the monodromy of $\phi_4$ and the paths used to identify it. The second row shows the skeleton used to identify the the fiber at $t=1$ with the standard surface $\Sigma_{2,2}$.}
    \label{fig:phi4}
\end{figure}

\smskip

\begin{lemma} \label{lem:phi4rotate} The factorization of $\phi_4 = \tau_1^2 \tau_2 \tau_3 \tau_4 \tau_5$ is Hurwitz equivalent to all of the other factorizations obtained by rotating the plane through a multiple of $2\pi/5$ and pulling back. Equivalently, the factorizations $\phi_4^k(\phi_4)$ and $\phi_4$ are Hurwitz equivalent. 
\end{lemma}
\begin{proof} The proof for this case is exactly analogous to the proof of the first case of Lemma~\ref{lem:phi2rotate} (the case of $\phi_2$).

\end{proof}

\noindent \textbf{Case} {\boldmath $y^2= x(x^5 +t^2)$}:
Splitting this singularity via $x(x^5 +t(t+s))$ yields two singularities, each of type $\phi_4$. By Lemma~\ref{lem:phi4rotate}, this yields a monodromy of $$\phi_4^2 = (\tau_1 \tau_2 \tau_3 \tau_4 \tau_5 \tau_5)^2.$$ This has the desired values of $b_1 = 0$ and $b_2 = 7$ and so the deformation of the corresponding resolution is flat. 
%subset of VIII-2 (\phi_1^3) missing 2 spheres
\smskip

\noindent \textbf{Case} {\boldmath $y^2= x(x^5 +t^3)$}: 
This splits by the deformation $y^2 = x(x^3(x^2-s)+t^3)$. We will call the positive word associated to the Lefschetz fibration $\phi_B$. The deformation produces a fibration with three singular fibers, which lie over $t=0$ and the six points $t= (\pm(4s(3s/5)^{1/2} - 6(3s/5)^{5/2})^{1/3}$, that lie on the six rays with angles $\theta = k*\pi/3, \, k = 0, \dots 5.$ The $t=0$ fiber has a type $x(x^3 + t^3)$ singularity which gives a pair of Dehn twists that separates off a genus 1 surface. For each of the other six singularities we take the paths from $t=1$ to the singular value and pull back the braid collision to get the factorization. These collisions occur between a pair of points which collide alternately at $x = 1$ and $x=-1$. Pulling these collisions back to the reference point at $t=1$ using the arcs in the $t$-plane as shown in Figure~\ref{fig:phiBmonodromy} we get the factorization $$\phi_4^3 = \phi_B = \tau_4 \tau_{a_2} \tau_3 \tau_5 \tau_2 \tau_4 \tau_{b_2} \tau_{b_2'}$$ where the labeling of the Dehn twists corresponds to the labels shown on the quotients under the hyperelliptic involution. This is a fibration of word length $8$ which yields a total space having $b_1 = 0$ and $b_2 = 3$, as needed. 
%subset of IX-1 (\phi_1^2) missing 2 spheres
\smskip

%\begin{figure}   \centering   \includegraphics[width=5in]{figs/Phi43.png}   \caption{The monodromy of $\phi_4^3$.}  \label{fig:phi43}\end{figure}
\begin{figure}
    \centering
    \includegraphics[width=4.7in]{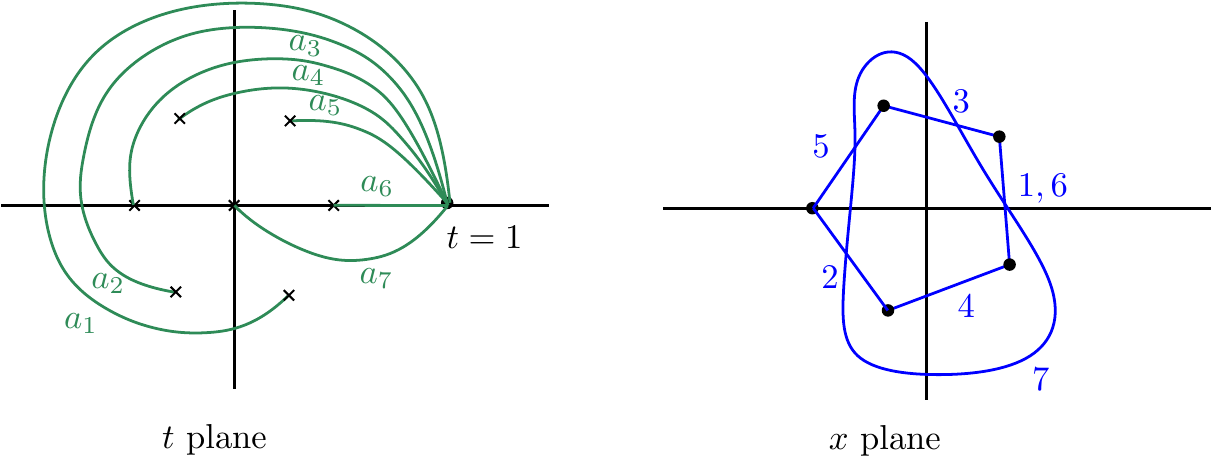} \\ 
     \vspace{0.2in}
    \includegraphics[width=2.7in]{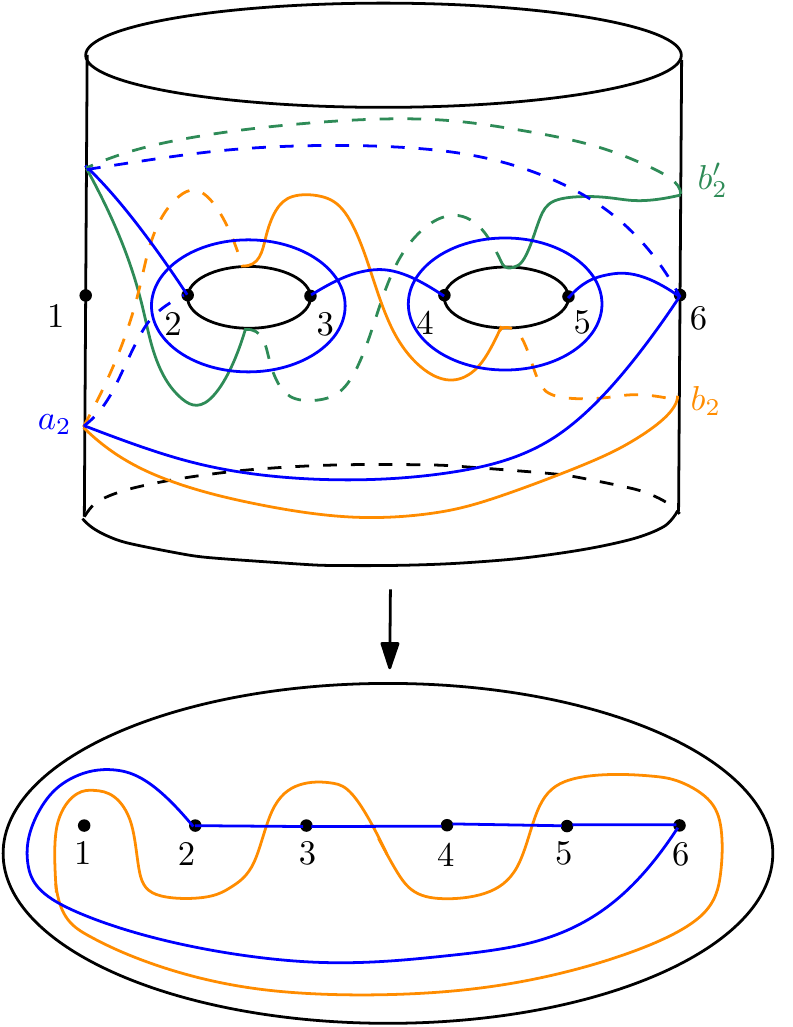}
    \caption{The the factorization $\phi_B$ of the monodromy of $\phi_4^3$.}
    \label{fig:phiBmonodromy}
\end{figure}
 
\noindent \textbf{Case} {\boldmath $y^2= x(x^5 +t^4)$}:
The resolution for this example has $b_1 = 0$ and $b_2 = 9$ and so we are looking for a factorization with word length 14. Splitting the singularity as $y^2= x(x^5 +t^3(t+s))$ yields two singularities. At $t=0$ we have the singularity of type $\phi_B$ (the previous case) of word length 8 and at $t=-s$ a singularity of type $\phi_4$ of word length 6. Choosing $s$ to be positive imaginary and choosing arcs so that the path from $t=1$ to $t=0$ lies on the real line and the path $t = - s$ is to the right, we pull the two factorizations from the two cases back to $t=1$. This gives us a positive word $\phi_3^4 = \phi_4^{k} (\phi_4) \phi_B$. By Lemma~\ref{lem:phi4rotate}, this factorization is Hurwitz equivalent to the factorization: $$\phi_4^4 = \phi_4 \phi_B = (\tau_1^2 \tau_2 \tau_3 \tau_4 \tau_5) (\tau_4 \tau_{a_2} \tau_3 \tau_5 \tau_2 \tau_4 \tau_{b_2} \tau_{b_2'}).$$
This has the correct word length and so gives a flat deformation of the resolution.

\smskip

\noindent \textbf{Case} {\boldmath $y^2= x(x^5 +t^5)$}:
In the resolution of $\phi_4^{5}$ we obtain a genus 2 surface of self intersection $-2$ (see the last row in Figure \ref{phi4ktable}). The resolution of this singularity then has $b_1 = 4$ and $b_2 = 1$, and so must be given by two Dehn twists along two homologous curves. The monodromy is then the boundary multitwist, $\phi_4^{5} = \tau_{\bdry}$, and the factorization must be the boundary multitwist as well.
%\clearpage
\bibliographystyle{amsalpha}
\bibliography{References}

\vspace{0.3in}
\end{document}